\documentclass{article}
\usepackage{fullpage}
\usepackage{graphicx,pgf}
\usepackage{amsmath}
\usepackage{amssymb}
\usepackage{amsthm}

\renewcommand{\theequation}{\arabic{section}.\arabic{equation}}

\newcommand{\re}{\operatorname{Re}}
\newcommand{\im}{\operatorname{Im}}
\newcommand{\e}{\operatorname{e}}
\newtheorem{lemma}{Lemma}[section]

\newtheorem{theorem}[lemma]{Theorem}

\title{Non-hexagonal lattices from a two species interacting system}
\author{Senping Luo
  \thanks{Department of Mathematics, University of British Columbia,
  Vancouver, BC, Canada  V6T 1Z2}
\and Xiaofeng Ren
\thanks{Department of Mathematics, The George Washington University,
  Washington, DC 20052, USA. Corresponding author, ren@gwu.edu}
\and Juncheng Wei \footnotemark[1]}

\begin{document}
\maketitle

\begin{abstract}
  A two species interacting system motivated by the density functional theory
  for triblock copolymers contains long range interaction that affects
  the two species differently. In a two species periodic assembly of
  discs, the two species appear alternately on a lattice.
  A minimal two species periodic  assembly is one with the least
  energy per lattice cell area.
  There is a parameter $b$ in $[0,1]$ and the type of the lattice associated
  with a minimal assembly varies depending on $b$. There are several
  thresholds defined by a number $B=0.1867...$ 
  If $b \in [0, B)$, a minimal assembly is associated with a rectangular
    lattice whose ratio of the longer side and the
    shorter side is in $[\sqrt{3}, 1)$;
  if $b \in [B, 1-B]$, a minimal assembly is associated with a square lattice;
  if $b \in (1-B, 1]$, a minimal assembly is associated with a
  rhombic lattice with an acute angle in $[\frac{\pi}{3}, \frac{\pi}{2})$.
  Only when $b=1$, this rhombic lattice is a hexagonal lattice. 
  None of the other values of $b$ yields a hexagonal lattice,
  a sharp contrast to the situation for one species
  interacting systems, where hexagonal lattices are ubiquitously observed.

  \


  \


\end{abstract}

\section{Introduction}
\setcounter{equation}{0}

From honeycomb to chicken wire fence, from graphene to carbon nanotube, 
the hexagonal pattern is ubiquitous in nature.
The honeycomb conjecture states that the hexagonal tiling is the best way
to divide a surface into regions of equal area with the least total perimeter
\cite{hales}. The Fekete problem minimizes an interaction energy of points
on a sphere and
obtains a hexagonal arrangement of minimizing points (with some defects
due to a topological reason) \cite{bendito}. 

Against this conventional wisdom,
we present a problem where the hexagonal pattern is {\em generally not}
the most favored structure. Our study is motivated  by
Nakazawa and Ohta's theory for triblock copolymer morphology
\cite{nakazawa, rw4}. In an ABC triblock
copolymer  a molecule is a subchain of type A monomers connected to a subchain of type B monomers which in turn is connected to a subchain of type C monomers. 
Because of  the repulsion between the unlike monomers, the different type subchains tend to segregate. However since subchains are chemically bonded in molecules, segregation cannot lead to a macroscopic phase separation; only micro-domains
rich in individual type monomers emerge, forming morphological phases. Bonding of distinct monomer subchains provides an inhibition mechanism in block
copolymers.

The mathematical study of the triblock copolmyer problem is still in the early
stage. There are existence theorems about stationary assemblies of
core-shells \cite{rwang}, double bubbles \cite{rw19}, and discs \cite{rwang2},
with the last work being the most relevant to this paper. 
Here we treat two of the three monomer types of a triblock copolymer
as species and view the third type as the surrounding environment, dependent on
the two species. This way a triblock copolymer is a two species interacting
system.

The definition of our two species interacting system starts with a
lattice $\Lambda$ on the complex plane generated by two nonzero complex
numbers $\alpha_1$ and $\alpha_2$, with $\im (\alpha_2/\alpha_1) >0$, 
\begin{equation}
  \label{Lambda}
  \Lambda = \{ j_1 \alpha_1 + j_2 \alpha_2: \ j_1, j_2 \in \mathbb{Z}\}. 
\end{equation}
Define by $D_\alpha$ the parallelogram cell 
\begin{equation}
  \label{cell}
  D_\alpha = \{ t_1 \alpha_1 + t_2 \alpha_2:  t_1,t_2 \in (0,1) \}
\end{equation}
associated to the basis $\alpha = (\alpha_1, \alpha_2)$
of the lattice $\Lambda$. The lattice $\Lambda$ defines an equivalence
relation on $\mathbb{C}$ where two complex numbers are equivalent if
their difference is in $\Lambda$. The resulting space of equivalent classes
is denoted $\mathbb{C} / \Lambda$. It can be represented by $D_\alpha$ where
the opposite edges of $D_\alpha$ are identified. 

There are two sets of parameters in our model. The first consists of two numbers
$\omega_1$ and $\omega_2$ satisfying
\begin{equation}
  0< \omega_1, \ \omega_2 <1, \ \mbox{and} \ \omega_1 + \omega_2 <1.
  \label{omega}
\end{equation}
The second is a two by
two symmetric matrix $\gamma$,
\begin{equation}
  \label{gamma}
  \gamma = \left [ \begin{array}{ll} \gamma_{11} & \gamma_{12}
      \\ \gamma_{21} & \gamma_{22} \end{array} \right ],
  \ \gamma_{12}=\gamma_{21}.
\end{equation}
Furthermore, in this paper we assume that
\begin{equation}
  \gamma_{11} >0, \ \gamma_{22}>0, \ \gamma_{12}\geq 0,
  \ \gamma_{11} \gamma_{22} - \gamma_{12}^2 \geq 0.
  \label{gamma-1}
\end{equation}

Our model is a variational problem defined on pairs of
$\Lambda$-periodic sets with prescribed average size. More specifically
a pair $(\Omega_1, \Omega_2)$ of two subsets of $\mathbb{C}$ is
admissible if the following conditions hold. 
    Both $\Omega_1$ and $\Omega_2$  are $\Lambda$-periodic, i.e.
    \begin{equation}
      \Omega_j + \lambda = \Omega_j, \ \mbox{for all} \ \lambda \in \Lambda,
\ j=1,2; \label{periodicity} \end{equation}
 $\Omega_1$ and $\Omega_2$ are disjoint in the sense that
\begin{equation}
  |\Omega_1 \cap \Omega_2| = 0;
 \label{disjoint}
\end{equation}
 the average size of $\Omega_1$ and $\Omega_2$
are fixed at $\omega_1$, $\omega_2 \in (0,1)$ respectively, i.e.
\begin{equation}
  \label{areaconstraint}
   \frac{|\Omega_j \cap D_\alpha |}{|D_\alpha|} = \omega_j, \ j=1,2.  
  \end{equation}

In \eqref{disjoint} and \eqref{areaconstraint}, $|\cdot|$ denotes the
two-dimensional
Lebesgue measure. 
Although it can be given in terms of $\alpha_1$ and $\alpha_2$,
$|D_\alpha|$ actually depends on the lattice $\Lambda$, not the particular basis $\alpha$,
and therefore we alternatively write it as $|\Lambda|$,
\begin{equation}
\label{latticearea}
 |\Lambda|=|D_\alpha| = \im (\overline{\alpha_1} \alpha_2).
\end{equation}

Given an admissible pair $(\Omega_1, \Omega_2)$, let
$\Omega_3=\mathbb{C} \backslash (\Omega_1 \cup \Omega_2)$.
Again $\Lambda$ imposes an equivalent relation on $\Omega_j$ and the
resulting space of equivalence classes is denoted $\Omega_j / \Lambda$,
$j=1,2,3$. Define a functional
${\cal J}_\Lambda$ to be the free energy of $(\Omega_1, \Omega_2)$ on a cell
given by
\begin{equation}
  \label{J}
        {\cal J}_\Lambda(\Omega_1, \Omega_2) = \frac{1}{2} \sum_{j=1}^3
        {\cal P}_{\mathbb{C}/\Lambda} (\Omega_j/\Lambda)   
  + \sum_{j,k=1}^2 \frac{\gamma_{jk}}{2}\int_{D_\alpha}
  \nabla I_\Lambda(\Omega_j)(\zeta) \cdot \nabla I_\Lambda(\Omega_k)(\zeta)
  \, d\zeta.
\end{equation}

In (\ref{J}), ${\cal P}_{\mathbb{C}/\Lambda}(\Omega_j/\Lambda)$, $j=1,2$,
is the perimeter of $\Omega_j/\Lambda$
in $\mathbb{C}/\Lambda$. One can take a representation $D_\alpha$ of
$\mathbb{C}/\Lambda$, with its opposite sides identified, and treat
$\Omega_j \cap D_\alpha$, also with points on opposite sides identified, as
a subset of $D_\alpha$. Then  ${\cal P}_{\mathbb{C}/\Lambda} (\Omega_j/\Lambda)$ is
the perimeter of $\Omega_j \cap D_\alpha$. If $\Omega_j \cap D_\alpha$ is bounded
by $C^1$ curves, then the perimeter is just the total length of the curves.
More generally, for a merely measurable $\Lambda$-periodic $\Omega_j$,
\begin{equation}
  \label{perimeter}
  {\cal P}_{\mathbb{C}/\Lambda} (\Omega_j/\Lambda)
  = \sup_g \Big \{ \int_{\Omega_j \cap D_\alpha} \mbox{div} \,g(x) \, dx: \ 
  g \in C^1(\mathbb{C}/\Lambda, \mathbb{R}^2), \ |g(x)| \leq 1
  \ \forall x \in \mathbb{C} \Big
 \}.
  \end{equation}
Here $g \in C^1(\mathbb{C}/\Lambda, \mathbb{R}^2)$ means that $g$
is a continuously differential, $\Lambda$-periodic vector field
on $\mathbb{C}$; $|g(x)|$ is the geometric norm of the vector
$g(x) \in \mathbb{R}^2$.

In $\sum_{j=1}^3{\cal P}_{\mathbb{C}/\Lambda} (\Omega_j/\Lambda)$
each boundary curve separating a $\Omega_j/\Lambda$ from a
$\Omega_k/\Lambda$, $j,k=1,2,3$, $j\ne k$,
is counted exactly twice. The constant $\frac{1}{2}$
in the front takes care of the double counting.

The function $I_\Lambda(\Omega_j)$ is the $\Lambda$-periodic
solution of Poisson's equation
\begin{equation}
  \label{I}
  - \Delta I_\Lambda(\Omega_j)(\zeta) = \chi_{\Omega_j}(\zeta) - \omega_j
   \ \mbox{in} \ \mathbb{C}, 
 \ \ \int_{D_\alpha} I_\Lambda(\Omega_j)(\zeta) \, d\zeta =0, 
\end{equation}
where $\chi_{\Omega_j}$ is the characteristic function of $\Omega_j$.
Despite the appearance, the functional ${\cal J}_\Lambda$ depends on the
lattice $\Lambda$ instead of the particular basis $\alpha$.

A stationary point $(\Omega_1,\Omega_2)$ of ${\cal J}_\Lambda$ 
is a solution to the following equations of a free boundary problem: 
\begin{align}
  \kappa_{13} + \gamma_{11} I_\Lambda(\Omega_1) + \gamma_{12} I_\Lambda(\Omega_2)
    &= \mu_1
  \ \ \mbox{on} \ \partial \Omega_1 \cap \partial \Omega_3 \label{euler1}
  \\ \kappa_{23} + \gamma_{12} I_\Lambda(\Omega_1) + \gamma_{22} I_\Lambda(\Omega_2)
  &= \mu_2
   \ \ \mbox{on} \  \partial \Omega_2 \cap \partial \Omega_3 
 \label{euler2} \\
 \kappa_{12} + (\gamma_{11}-\gamma_{12}) I_\Lambda(\Omega_1) + 
 (\gamma_{12}-\gamma_{22}) I_\Lambda(\Omega_2)
 &= \mu_1-\mu_2 \ \ \mbox{on} \ 
  \partial \Omega_1 \cap \partial \Omega_2 \label{euler3} \\
   T_{13} + T_{23} + T_{12} &= \vec{0} \ \ \mbox{at} \ 
    \partial \Omega_1 \cap \partial \Omega_2 \cap \partial \Omega_3. 
  \label{euler4} 
\end{align}
In (\ref{euler1})-(\ref{euler3}) $\kappa_{13}$, $\kappa_{23}$, and $\kappa_{12}$
are the curvatures of the curves $\partial \Omega_1 \cap \partial \Omega_3$, 
$\partial \Omega_2 \cap \partial \Omega_3$, and 
$\partial \Omega_1 \cap \partial \Omega_2$, respectively.
The unknown constants $\mu_1$ and $\mu_2$ are Lagrange multipliers
associated with the constraints \eqref{areaconstraint} for $\Omega_1$ and
$\Omega_2$ respectively.
The three interfaces,  $\partial \Omega_1 \cap \partial \Omega_3$,
$\partial \Omega_2 \cap \partial \Omega_3$ and
$\partial \Omega_1 \cap \partial \Omega_2$,
may meet at a common point in $D$, which is termed a triple junction point. 
In (\ref{euler4}), $T_{13}$, $T_{23}$ and $T_{12}$ are respectively
the unit tangent vectors
of these curves at triple junction points. This equation simply says that at
a triple junction
point three curves meet at $\frac{2\pi}{3}$ angle.

In this paper, we only consider a special type of $(\Omega_1, \Omega_2)$,
termed two species periodic assemblies of discs, denoted by 
$(\Omega_{\alpha,1}, \Omega_{\alpha,2})$, with 
\begin{align}
  \Omega_{\alpha, 1}& = \bigcup_{\lambda \in \Lambda}
   \Big \{ B(\xi, r_1) \cup B(\xi', r_1):
  \xi = \frac{3}{4} \alpha_1
     + \frac{1}{4}\alpha_2 + \lambda, \ 
     \xi' = \frac{1}{4} \alpha_1
     + \frac{3}{4} \alpha_2 + \lambda \Big \}, 
  \label{Omegaalpha1} \\
  \Omega_{\alpha, 2}& = \bigcup_{\lambda \in \Lambda}
  \Big \{ B(\xi, r_2) \cup B(\xi', r_2):
  \xi = \frac{1}{4} \alpha_1
     + \frac{1}{4} \alpha_2 + \lambda , \ 
     \xi' = \frac{3}{4}  \alpha_1
     + \frac{3}{4} \alpha_2 + \lambda
     \Big \}.   \label{Omegaalpha2}
\end{align}
In \eqref{Omegaalpha1} and \eqref{Omegaalpha2},
$B(\xi, r_j)$, or $B(\xi',r_j)$, is the closed disc centered at
$\xi$ of radius $r_j$; the $r_j$'s are given by
\begin{equation}
  \omega_j = \frac{2\pi r_j^2}{|D_\alpha|}, \ j=1,2.
  \label{rj}
\end{equation}
Be aware that $(\Omega_{\alpha,1},\Omega_{\alpha,2})$ defined this way
depends on the basis $\alpha$, not
the lattice $\Lambda$ generated by $\alpha$. One may have two different
bases that generate the same lattice, but they define two distinct
assemblies.

Shifting $(\Omega_{\alpha,1}, \Omega_{\alpha,2})$ does not
change its energy, so
our choice for the centers of the discs
in \eqref{Omegaalpha1} and \eqref{Omegaalpha2} is not the only one.
Another aesthetically pleasing placement is to put the disc centers on
the lattice points and half lattice points; see Figure \ref{f-assemblies}.
Nevertheless we prefer not to have discs on the boundary of the parallelogram
cell $D_\alpha$.

\begin{figure}
\centering
 \includegraphics[scale=0.13]{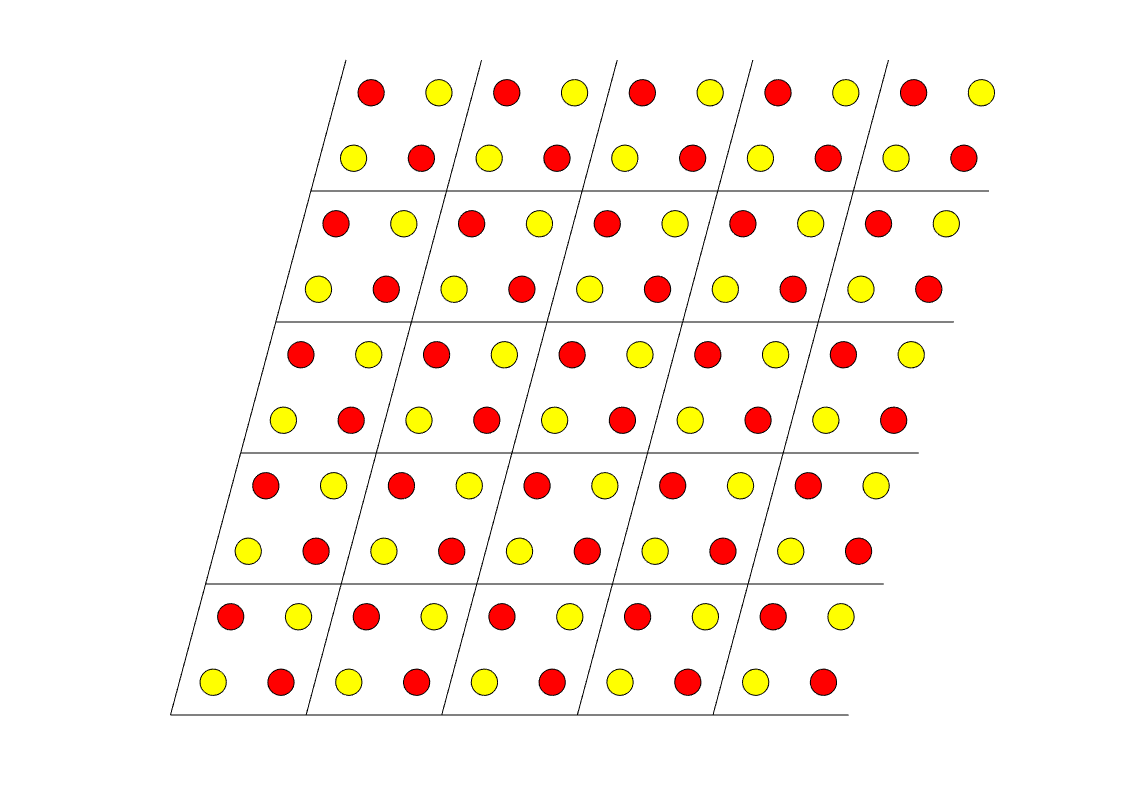} 
 \includegraphics[scale=0.13]{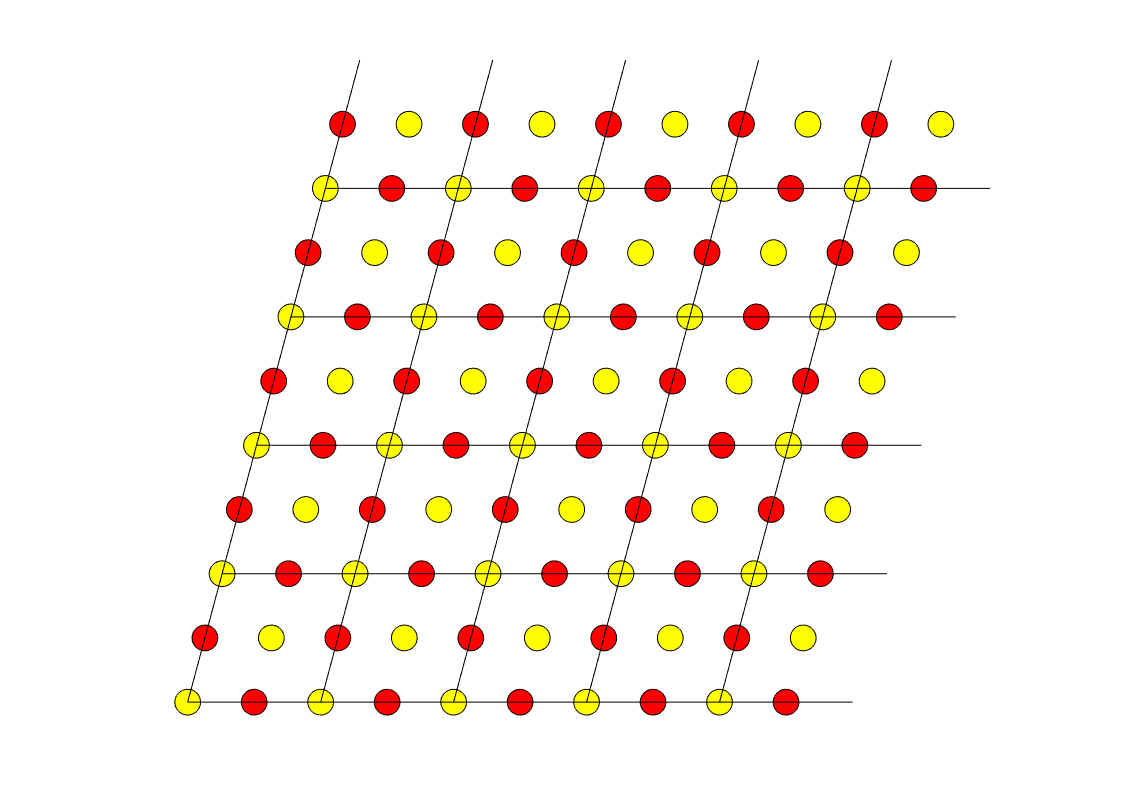}
 \caption{A two species periodic assembly of discs given by
   \eqref{Omegaalpha1} and \eqref{Omegaalpha2}, and a shift of the  assembly
   with disc centers at the lattice points and the half lattice
   points.}
\label{f-assemblies}
\end{figure}

A two species periodic assembly $(\Omega_{\alpha,1}, \Omega_{\alpha,2})$
is  not a stationary point of the energy
fuctional ${\cal J}_\Lambda$. However Ren and Wang have shown the existence
of statoinary points that are unions of perturbed discs in a bounded
domain with the Neumann bounary condition \cite{rwang2}.
Numerical evidence strongly suggests the existence of
stationary points similar to two species assemblies \cite{wang-ren-zhao}.

In this paper we determine, in terms of $\alpha$,
which $(\Omega_{\alpha,1}, \Omega_{\alpha,2})$ is the most energetically
favored. For this purpose, it is more appropriate to consider the 
energy per cell area instead of the energy on a cell. Namely consider
\begin{equation}
  \label{tJ}
  \widetilde{\cal J}_\Lambda(\Omega_1,\Omega_2)
  = \frac{1}{|\Lambda|} {\cal J}_\Lambda(\Omega_1,\Omega_2),
\end{equation}
take $(\Omega_1,\Omega_2)$ to be a two species periodic assembly, and
minimize energy per cell area among all such assemblies
with respect to $\alpha$, i.e.
\begin{equation}
  \min_\alpha \Big \{  \widetilde{\cal J}_\Lambda
   (\Omega_{\alpha,1}, \Omega_{\alpha,2}):
  \ \alpha=(\alpha_1,\alpha_2), \
  \alpha_1, \alpha_2 \in \mathbb{C}\backslash\{0\},
    \ \im \frac{\alpha_2}{ \alpha_1} >0, 
  \ \Lambda \ \mbox{is generated by}
   \ \alpha \Big \}.
  \label{minimization}
\end{equation}

Several lattices will appear as the most favored structures. They are
illustrated in Figure \ref{f-lattices}. A rectangular
lattice has a basis $\alpha$ whose parallelogram cell $D_\alpha$ is a rectangle.
A square lattice has a square as a parallelogram cell. A rhombic lattice
has a rhombus cell,
i.e. a parallelogram cell whose four sides have the same length. Finally
a hexagonal lattice has a parallelogram cell with 
four equal length sides and an angle of $\frac{\pi}{3}$ between two sides. 
If we let
\begin{equation}
  \tau = \frac{\alpha_2}{\alpha_1},
\end{equation}
then in terms of $\tau$,
$\Lambda$ is rectangular if $\re \tau =0$, $\Lambda$ is
square if $\tau = i$, $\Lambda$ is rhombic if $|\tau| =1$, and
$\Lambda$ is hexagonal if $\tau = \frac{1}{2} + \frac{\sqrt{3}}{2} i$. 
Note that these classes of lattices are not mutually exclusive. 
A hexagonal lattice is a rhombic lattice; a square
lattice is both a rectangular lattice and a rhombic lattice.

The reason that a rhombic lattice with a $\frac{\pi}{3}$ angle is termed
a hexagonal lattice comes from its Voronoi cells. At each
lattice point, the Voronoi cell of this lattice point consists of points
in $\mathbb{C}$ that are closer to this lattice point than any
other lattice points.
For the rhombic lattice with a $\frac{\pi}{3}$ angle, the Voronoi
cell at each lattice point is a regular hexagon. With Voronoi cells at
all lattice points, the hexagonal lattice gives rise to a honeycomb pattern. 

The main result of this paper asserts that for a two species
periodic assembly of discs to minimize the energy per cell area,
its associated parallelogram cell is either a
rectangle (including a square) whose ratio of the longer side
and the shorter side lies between $1$ and $\sqrt{3}$, or a rhombus
(including one with a $\frac{\pi}{3}$ acute angle)
whose acute angle is between $\frac{\pi}{3}$ and
$\frac{\pi}{2}$. Any two species
periodic assembly of discs that minimizes the energy per cell area is
called a minimal assembly and its  associated lattice 
is called a minimal lattice.

The most critical parameter in this problem is $b$ given in terms of
$\omega_j$ and $\gamma_{jk}$ by
\begin{equation}
        b =  \frac{2 \gamma_{12} \omega_1\omega_2}
       { \gamma_{11} \omega_1^2 + \gamma_{22} \omega_2^2}. 
    \label{binit}
\end{equation}
Conditions \eqref{omega}, \eqref{gamma}, and \eqref{gamma-1}
on $\omega_j$ and $\gamma_{jk}$ imply that
\begin{equation}
  b \in [0,1].
  \label{b-range}
  \end{equation}

To ensure the disjoint condition \eqref{disjoint}
for potential minimal assemblies we assume that
$\omega_1$ and $\omega_2$ are sufficiently small. Namely let $\omega_0>0$
be small enough so that if
\begin{equation}
  \omega_j < \omega_0, \ \ j=1,2,
  \label{omega0}
\end{equation}
and $(\Omega_{\alpha, 1}, \Omega_{\alpha,2})$ is a
two species periodic assembly of discs
whose basis $(\alpha_1,\alpha_2)$ satisfies
\begin{equation}
  \label{rectangle-cond}
  \re \tau =0 \ \mbox{and} \  |\tau| \in [1, \sqrt{3}],
\end{equation}
or 
\begin{equation}
  \label{rhombus-cond}
  |\tau| =1 \ \mbox{and} \  \arg \tau \in
  \Big [\frac{\pi}{3}, \frac{\pi}{2} \Big ],
\end{equation}
then $(\Omega_{\alpha, 1}, \Omega_{\alpha,2})$ is disjoint in the sense
of \eqref{disjoint}.
The line segment \eqref{rectangle-cond} and the arc \eqref{rhombus-cond}
are illustrated in the first plot of Figure \ref{f-W}. 
Now we state our theorem.

\begin{theorem}
  Let the parameters $\omega_j$, $j=1,2$, and $\gamma_{jk}$, $j,k=1,2$,
    satisfy the conditions \eqref{omega}, \eqref{gamma}, \eqref{gamma-1},
    and \eqref{omega0}. 
  The minimization problem \eqref{minimization} always admits a minimum. Let  
  $\alpha_\ast=(\alpha_{\ast,1}, \alpha_{\ast,2})$ be a minimum of
  \eqref{minimization}, $\Lambda_\ast$ be the lattice determined by
  $\alpha_\ast$.
   Then there exists $B = 0.1867...$ such that the following statements hold.
  \begin{enumerate}
  \item  If $b=0$, then $\Lambda_\ast$ is a rectangular lattice
    whose ratio of the longer side and the
    shorter side is $\sqrt{3}$.
  \item If $b \in (0, B)$, then $\Lambda_\ast$ is a rectangular lattice
    whose ratio of the longer side  and the shorter side  is
    in $(1, \sqrt{3})$. As $b$ increases from $0$ to $B$, this ratio
    decreases from $\sqrt{3}$ to $1$.
  \item If $b \in [B, 1-B]$, then $\Lambda_\ast$ is a square lattice.
  \item If $b \in (1-B, 1)$, then $\Lambda_\ast$ is a non-square,
    non-hexagonal rhombic lattice with an acute angle in
    $(\frac{\pi}{3}, \frac{\pi}{2})$. As $b$ increases from
    $1-B$ to $1$, this angle decreases from $\frac{\pi}{2}$ to
    $\frac{\pi}{3}$. 
  \item If $b =1$, then $\Lambda_\ast$ is a hexagonal lattice. 
    \end{enumerate}
  \label{t-main}
\end{theorem}

The threshold $B$ is defined precisely in \eqref{B} by two infinite
series, from which one finds its numerical value.

Only in the case $b=1$, $\widetilde{\cal J}_\Lambda
(\Omega_{\alpha,1}, \Omega_{\alpha,2})$ is minimized by a hexagonal lattice. In
all other cases minimal lattices are not hexagonal. As a matter of fact,
our assumption on $\gamma$ in this paper is a bit different from the
conditions for
$\gamma$ in a triblock copolymer. In a triblock copolymer, instead of
\eqref{gamma-1}, $\gamma$ needs to be positive definite.
In \cite{rwang2}, where Ren and Wang found  assemblies of perturbed discs
as stationary points, $\gamma_{12}$ is positive.
In terms of $b$, $\gamma$ being positive
definite and $\gamma_{12}>0$ mean that $b \in (0,1)$. 

In this paper we include both the $b=0$ case and the $b=1$ case
for good reasons.
The case $b=1$ corresponds to $\gamma_{11} \gamma_{22} - \gamma_{12}^2=0$,
i.e. $\gamma$ has a non-trivial kernel, and
$(-\omega_1, \omega_2)$ is in the kernel of $\gamma$. 
This case is actually very special. It is
equivalent to a problem studied by
Chen and Oshita in \cite{chenoshita-2}, a simpler one species analogy of
the two species problem studied here. The motivation of that problem
comes from the study of diblock copolymers where a 
molecule is a subchain of type A monomers connected to a subchain of
type B monomers. With one type treated as a species and the other as
the surrounding environment, a diblock copolymer is a one species
interacting system.

The recent years have seen active work on the diblock copolymer problem;
see \cite{rw, rw12, aco,
  choksi-peletier, morini-sternberg, goldman-muratov-serfaty}
and the references therein. 
Based on a density functional theory of Ohta and Kawasaki
\cite{ok}, the free energy of a diblock copolymer system on a 
$\Lambda$-periodic domain is
\begin{equation}
  {\cal E}_\Lambda (\Omega) = {\cal P}_{\mathbb{C}/\Lambda}(\Omega/\Lambda)
  + \frac{\gamma_d}{2} \int_{D_\alpha} |\nabla I_\Lambda(\Omega)(\zeta)|^2 \,
  d \zeta.
  \end{equation}
Here, analogous to the two species problem, $\Omega$ is a $\Lambda$-periodic
subset of $\mathbb{C}$ under the average area constraint
  $\frac{|\Omega \cap D_\alpha|}{|D_\alpha|} = \omega$
where $\omega \in (0,1)$ is one of the two given parameters.
The other parameter is
the number $\gamma_d >0$. Now take $\Omega$ to be $\Omega^d_\Lambda$,
the union of discs centered at
$\frac{\alpha_1+\alpha_2}{2} + \lambda$, $\lambda \in \Lambda$,
of radius $\sqrt{\frac{\omega |D_\alpha|}{\pi}}$,
and minimize
the energy per cell area with respect to $\Lambda$:
\begin{equation}
  \label{db-E}
   \min_\Lambda \frac{1}{|\Lambda|} {\cal E}_\Lambda(\Omega^d_\Lambda)
\end{equation}
This time, unlike in the two species problem, $\Omega^d_\Lambda$ depends on
the lattice $\Lambda$, not the basis $(\alpha_1, \alpha_2)$. 
Chen and Oshita showed that \eqref{db-E} is minimized by a hexagonal lattice.

In \cite{sandier-serfaty} Sandier and Serfaty studied the
Ginzburg-Landau problem with magnetic field and arrived at a reduced
energy. Minimization of this energy turns out to be the same as
the minimization problem \eqref{db-E}. 

In our two species problem \eqref{minimization},
the condition $b=1$ actually makes
the two species indistinguishable as far as interaction is concerned. 
It means that the two species function as one species, hence the equivalence
to the one species problem \eqref{db-E}.
The case $b=0$ is {\em dual} to the $b=1$ case, a point explained below.
It is therefore natural to include both cases.

Our work starts with Lemma \ref{l-size}, which states that for us
to solve \eqref{minimization} it suffices to minimize the energy among 
two species periodic assemblies of unit cell area.
Then in Lemma \ref{l-Ff} it is
shown that the latter problem is equivalent to maximizing a function,
\begin{equation}
  f_b (z) = b \log \big | \im \big (z \big ) \eta\big (z \big) \big|
  + (1-b) \log \big | \im \big(\frac{z+1}{2}\big)
  \eta \big(\frac{z+1}{2}\big) \big |,
    \label{fb-intro}
\end{equation}
with respect to $z$ in the set $\{ z \in \mathbb{C}: \ \im z >0, 
\ |z| \geq 1, \ 0 \leq \re z \leq 1 \}$. Here
\begin{equation}
  \label{eta-intro}
  \eta(z) = e^{\frac{\pi}{3} z i}\prod_{n=1}^\infty\big(1- e^{2\pi nz i}\big)^4
\end{equation}
is the fourth power of the Dedekind eta function.

If $b=1$, then $f_b=f_1$ and we are looking at the problem studied by
Chen and Oshita \cite{chenoshita-2}, and Sandier and Serfaty
\cite{sandier-serfaty}. In this case, $f_1$ is maximized in a smaller set,
$\{ z \in \mathbb{C}: \ |z| \geq 1, \ 0 \leq \re z \leq 1/2 \}$.
Using a maximum principle argument, Chen and Oshita showed that $f_1$
is maximized at $z = \frac{1}{2} + \frac{\sqrt{3}}{2} i$, which
corresponds to the hexagonal lattice. Sandier and Serfaty used a relation
between the Dedekind eta function and the Epstein zeta function, and
a property of the Jacobi theta function to arrive at the same conclusion. 

Neither method seems to be applicable to the two species system with
$b \ne 1$. Instead we rely on a duality principle, Lemma \ref{l-dual},
which shows that maximizing $f_b$ is equivalent to maximizing $f_{1-b}$.
This allows us to only consider $b \in [0,1/2]$, and there we are
able to show that $f_b(z)$ attains the maximum on the imaginary
axis above $i$, i.e. $\re z =0$ and $\im z \geq 1$. 

So we turn to maximize $f_b(y i)$ with respect to $y \geq 1$.
The most technical part of this work, Lemma \ref{l-fb-im}, shows that when
$b =0$, $f_0(y i)$ is maximized at $y=\sqrt{3}$;
when $b \in (0, B)$, $f_b(y i)$ is maximized at some
$y = q_b \in (1, \sqrt{3})$;
when $b \in [B, 1]$, $f_b(y i)$ is maximized at $y=1$.
The theorem then follows readily.
The key step in the proof of Lemma \ref{l-fb-im}
is to establish a  monotonicity property
for the ratio of the derivatives of $f_0(y i)$ and $f_1(y i)$ with respective
to $y \in (1, \sqrt{3})$. This piece of argument is placed in the appendix
so a reader who is more interested in the overall strategy of this work
may skip it at the first reading.

\begin{figure}
\centering
 \includegraphics[scale=0.12]{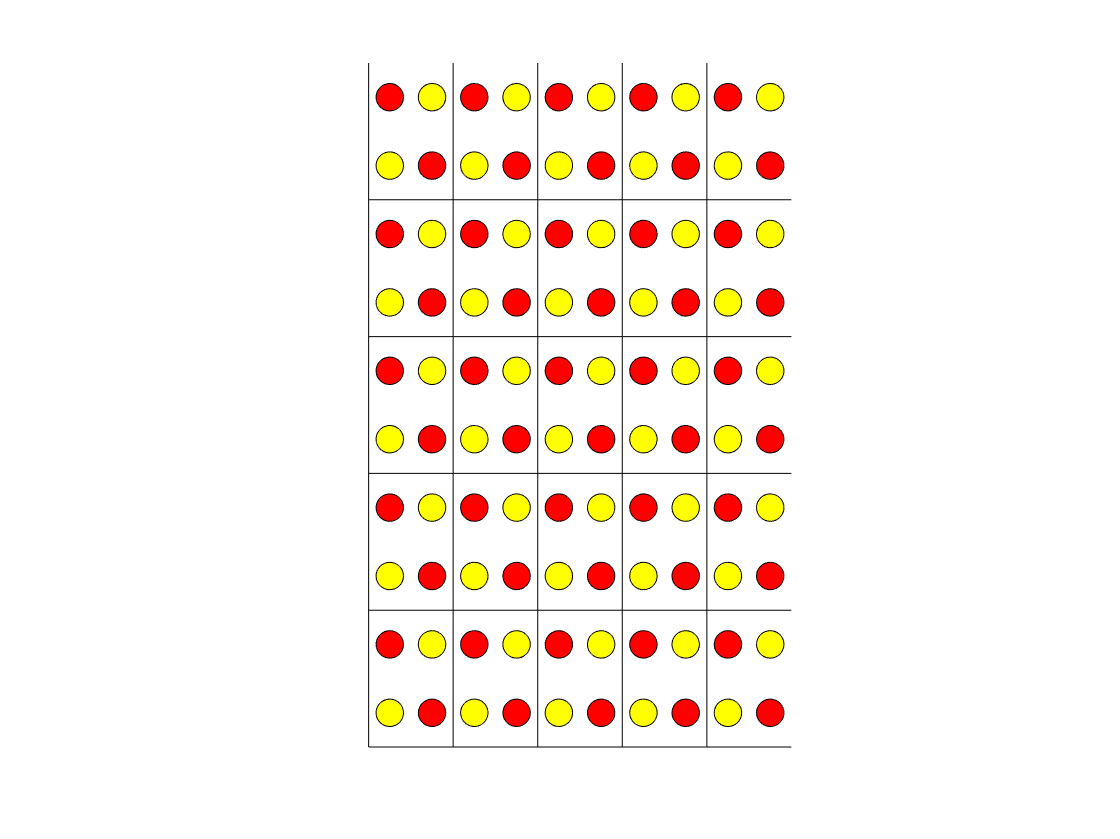} 
 \includegraphics[scale=0.12]{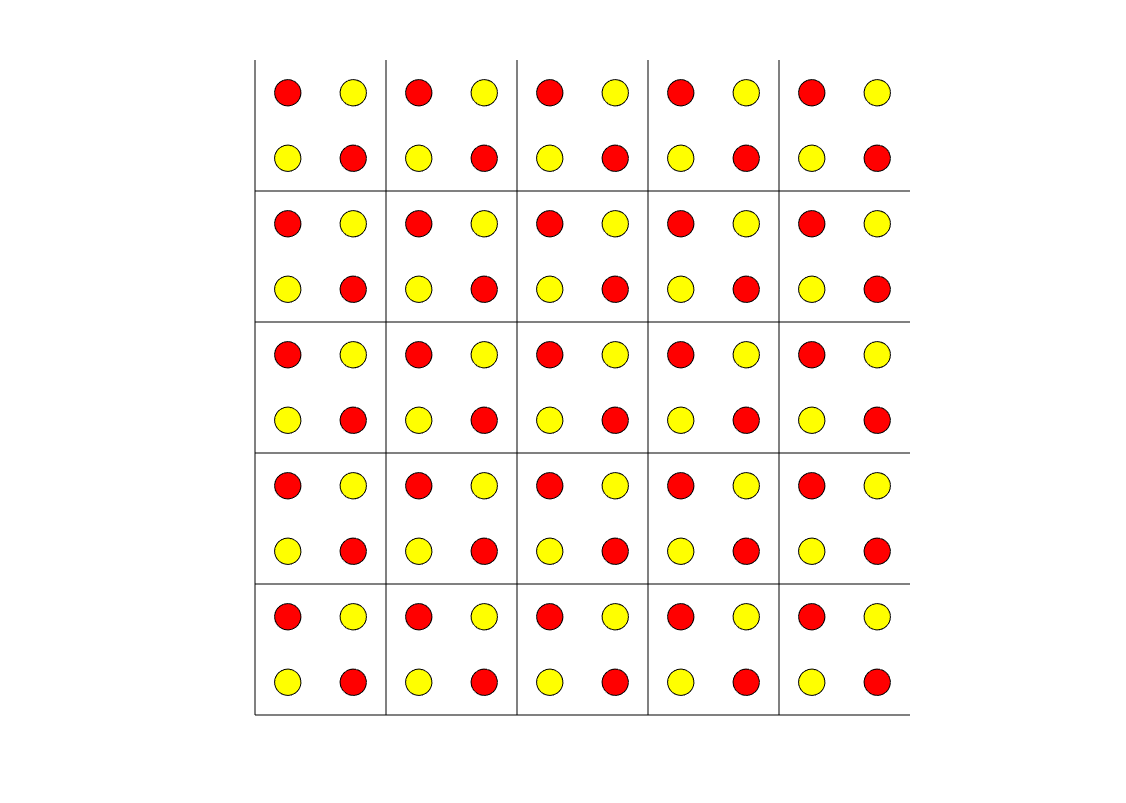}
 \includegraphics[scale=0.12]{rhombus.png} 
 \includegraphics[scale=0.15]{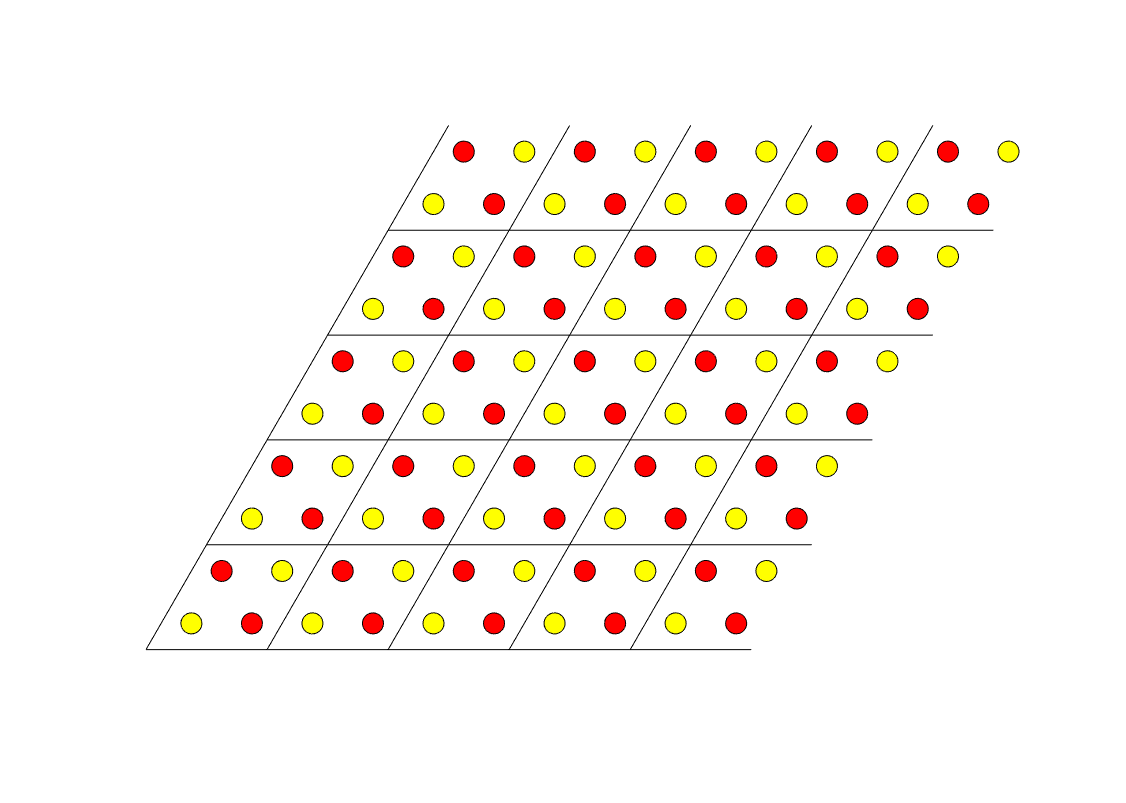}
 \caption{Two species periodic assemblies of discs with their associated
   lattices.
   First row from left to right: a rectangular lattice and a square lattice.
   Second row from left to right: a rhombic lattice and a hexagonal lattice.}
\label{f-lattices}
\end{figure}

\section{Derivation of $f_b$}
\setcounter{equation}{0}

The size and the shape of a two species periodic assembly play different
roles in its energy. To separate the two factors write the basis
of a given two species periodic assembly as
$t \alpha = (t\alpha_1, t\alpha_2)$ where $ t \in (0, \infty)$ and
the parallelogram generated by $\alpha = (\alpha_1, \alpha_2)$ has the unit
area, i.e. $|D_\alpha| =1$. This way the assembly is now denoted by
$(\Omega_{t\alpha,1},\Omega_{t\alpha,2}) $, with $t$ measuring the size of
the assembly (note $|D_{t\alpha}| = t^2$), and $D_\alpha$ describing
the shape of the assembly. 
The lattice generated by $\alpha$ is denoted $\Lambda$ and the
lattice generated by $t \alpha$ is $t \Lambda$. 

\begin{lemma}
  Fix  $\alpha_1, \alpha_2 \in
  \mathbb{C}\backslash \{0\}$, $\im(\alpha_2/\alpha_1)>0$, and
  $|D_\alpha|=1$.
  Among all 2 species periodic assemblies $\Omega_{t\alpha}$,
  $t \in (0,\infty)$, the energy per cell area is minimized by the one with
  $t = t_{\alpha}$, where
  \begin{equation}
    \label{t-alpha}
  t_\alpha = \Big ( \frac{2 \sqrt{2\pi \omega_1} + 2 \sqrt{2\pi \omega_2} }
  {\sum_{j,k=1}^2 \gamma_{jk} \int_{D_\alpha}
  \nabla I_{\Lambda}(\Omega_{\alpha,j})(\zeta)
  \cdot \nabla I_{\Lambda}(\Omega_{\alpha, k})(\zeta) \, d \zeta } \Big )^{1/3}.
  \end{equation}
  The energy per cell area of this assembly is
  \begin{equation}
    \label{energy-per-cell}
  \widetilde{\cal J}_{t_\alpha \Lambda}
  ((\Omega_{t_\alpha \alpha,1}, \Omega_{t_\alpha \alpha,2} ) =
  3   \Big ( \sqrt{2\pi \omega_1} +  \sqrt{2\pi \omega_2} \Big )^{2/3}
  \Big ( \sum_{j,k=1}^2 \frac{\gamma_{jk}}{2} \int_{D_\alpha}
  \nabla I_{\Lambda}(\Omega_{\alpha,j})(\zeta)
  \cdot \nabla I_{\Lambda}(\Omega_{\alpha, k})(\zeta) \, d \zeta \Big )^{1/3}.
  \end{equation}
  Consequently minimization of \eqref{minimization} is reduced to minimizing
  \begin{equation}
  \label{F}
        {\cal F}(\alpha) = \sum_{j,k=1}^2 \frac{\gamma_{jk}}{2}
        \int_{D_\alpha}
  \nabla I_{\Lambda}(\Omega_{\alpha,j})(\zeta)
  \cdot \nabla I_{\Lambda}(\Omega_{\alpha, k})(\zeta) \, d \zeta, \ \ |D_\alpha| =1,
  \end{equation}
with respect to $\alpha$ of unit cell area.
  \label{l-size}
  \end{lemma}

\begin{proof}
Between the two lattices, the functions
$I_{t\Lambda}(\Omega_{t\alpha,j})$ and $I_{\Lambda}(\Omega_{\alpha,j})$
are related by
\begin{equation}
  \label{scale-I}
  I_{t\Lambda}(\Omega_{t\alpha,j})(\chi) = t^2 I_{\Lambda}(\Omega_{\alpha,j})(\zeta),
  \ \ t \zeta=\chi, \ \zeta, \chi \in \mathbb{C}
\end{equation}
because of the equation \eqref{I}. Then
\begin{align*}
  {\cal J}_{t\Lambda}(\Omega_{t\alpha,1},\Omega_{t \alpha,2} ) & =
  t \Big ( 2 \sqrt{2\pi \omega_1} + 2 \sqrt{2\pi \omega_2} \Big )
  + \sum_{j,k=1}^2 \frac{\gamma_{jk}}{2} \int_{D_{t\alpha}}
  \nabla I_{t\Lambda}(\Omega_{t\alpha,j}) (\chi) \cdot
  \nabla I_{t\Lambda}(\Omega_{t\alpha,k})(\chi) \, d \chi \\
  &=  t \Big ( 2 \sqrt{2\pi \omega_1} + 2 \sqrt{2\pi \omega_2} \Big )
  + t^4 \sum_{j,k=1}^2 \frac{\gamma_{jk}}{2} \int_{D_\alpha}
  \nabla I_{\Lambda}(\Omega_{\alpha,j})(\zeta)
  \cdot \nabla I_{\Lambda}(\Omega_{\alpha, k})(\zeta) \, d \zeta.
  \end{align*}
The energy per cell area is
\begin{align*}
  \nonumber
   \widetilde{\cal J}_{t\Lambda}(\Omega_{t\alpha,1}, \Omega_{t \alpha,2}) & = 
  \Big (\frac{1}{t} \Big )^2 {\cal J}_{t\Lambda} (\Omega_{t\alpha}) \\ & =
   \frac{1}{t} \Big ( 2 \sqrt{2\pi \omega_1} + 2 \sqrt{2\pi \omega_2} \Big )
  + t^2 \sum_{j,k=1}^2 \frac{\gamma_{jk}}{2} \int_{D_\alpha}
  \nabla I_{\Lambda}(\Omega_{\alpha,j})(\zeta)
  \cdot \nabla I_{\Lambda}(\Omega_{\alpha, k})(\zeta) \, d \zeta.
\end{align*}
With respect to $t$, the last quantity is minimized at $t = t_\alpha$
given in \eqref{t-alpha}, and the minimum value is given in
\eqref{energy-per-cell}.

Later one needs to minimize the right side of \eqref{energy-per-cell}
with respect to $\alpha$, $|D_\alpha|=1$. This is equivalent to minimze
${\cal F}(\alpha)$ 
with respect to $\alpha$, $|D_\alpha|=1$. Once a minimum, say $\alpha_\ast$,
is found, then compute $t_{\alpha_\ast}$ from \eqref{t-alpha} and make
the assembly $\Omega_{t_{\alpha_\ast} \alpha_\ast}$ with the basis
$t_{\alpha_\ast} \alpha_\ast$. This assembly minimizes \eqref{minimization}. 
\end{proof}

Now that the minimization problem \eqref{minimization} is reduced to
minimizing ${\cal F}$, we proceed to simplify ${\cal F}(\alpha)$ to a more
amenable form. To this end, one expresses 
the solution of \eqref{I}  in terms of the Green's function
$G_\Lambda$ of the $-\Delta$ operator as
\begin{equation}
  I_{\Lambda}(\Omega_j)(\zeta) = \int_{\Omega_j \cap D_\alpha}
  G_\Lambda(\zeta - \chi) \, d\chi. 
  \label{I2}
  \end{equation}
Here $G_\Lambda$ is the $\Lambda$-periodic solution of
\begin{equation}
  - \Delta G_\Lambda (\zeta) = \sum_{\lambda \in \Lambda}
   \delta_{\lambda} (\zeta) - \frac{1}{|\Lambda|},
  \ \int_{D_\alpha} G_\Lambda(\zeta) \, d\zeta =0. 
  \label{G}
\end{equation}
In \eqref{G} $\delta_\lambda$ is the delta measure at $\lambda$, and
$D_\alpha$ is a parallelogram cell of $\Lambda$. It is known that
\begin{align}
  \nonumber
  G_\Lambda(\zeta) & = \frac{|\zeta|^2}{4 |\Lambda|}
  - \frac{1}{2\pi} \log
  \Big |
  \e\Big (\frac{\zeta^2 \overline{\alpha_1}}{4 i |\Lambda| \alpha_1}
  -\frac{\zeta}{2\alpha_1} +\frac{\alpha_2}{12\alpha_1} \Big )
  \Big (1-\e(\frac{\zeta}{\alpha_1}) \Big ) \\ & \qquad
  \prod_{n=1}^\infty\Big ( \big( 1- \e(n\tau + \frac{\zeta}{\alpha_1}) \big )
  \big (1- \e(n\tau - \frac{\zeta}{\alpha_1}) \big) \Big )
  \Big |  \label{Green0}
\end{align}
A simple proof of this fact can be found in \cite{chenoshita-2}.
Throughout this paper one writes
\begin{equation}
  \e(z) = e^{2\pi i z}
  \label{e}
\end{equation}
and
\begin{equation}
  \tau = \frac{\alpha_2}{\alpha_1}.
  \label{tau}
  \end{equation}
Sometimes one singles out the singularity of $G_\Lambda$ at $0$ and decompose 
$G_\Lambda$ into 
\begin{equation}
  \label{Green}
  G_\Lambda(\zeta) = -\frac{1}{2\pi} \log \frac{2\pi |\zeta|}{\sqrt{|\Lambda|}}
  + \frac{|\zeta|^2}{4 |\Lambda|} + H_\Lambda(\zeta)
\end{equation}
where
\begin{align}
  \nonumber
  H_\Lambda(\zeta) & = - \frac{1}{2\pi} \log
  \Big |
  \e\Big (\frac{\zeta^2 \overline{\alpha_1}}{4 i |\Lambda| \alpha_1}
  -\frac{\zeta}{2\alpha_1} +\frac{\alpha_2}{12\alpha_1} \Big )
  \frac{\sqrt{|\Lambda|}}{2\pi \zeta}
  \Big (1-\e(\frac{\zeta}{\alpha_1}) \Big ) \\ & \qquad
  \prod_{n=1}^\infty\Big ( \big( 1- \e(n\tau + \frac{\zeta}{\alpha_1}) \big )
  \big (1- \e(n\tau - \frac{\zeta}{\alpha_1}) \big) \Big )
  \Big |  \label{H}
\end{align}
is a harmonic function on $(\mathbb{C} \backslash \Lambda)
\cup \{ 0\} $. 

The integral term in \eqref{J} can be written in several different ways: 
\begin{align}
   \int_{D_\Lambda}
  \nabla I_\Lambda(\Omega_j)(\zeta) \cdot \nabla I_\Lambda(\Omega_k)(\zeta)
  \, d\zeta & = \int_{\Omega_k} I_\Lambda(\Omega_j)(\zeta) \, d \zeta
   = \int_{\Omega_j} I_\Lambda(\Omega_k)(\chi) \, d \chi \nonumber \\
   & = \int_{\Omega_k} \int_{\Omega_j} G_\Lambda(\zeta - \chi) \, d\chi d\zeta.
   \label{rewrite}
  \end{align}

\begin{lemma}
  \label{l-FtF}
Minimizing
${\cal F}(\alpha)$  with respect to $\alpha$ of unit cell area
is equivalent to minimizing 
\begin{equation}
  \label{Ftilde}
  \widetilde{\cal F} (\alpha)
  = H_\Lambda(0) + G_\Lambda\Big(\frac{\alpha_1+\alpha_2}{2}\Big)   +
  b \Big( G_\Lambda\Big(\frac{\alpha_1}{2}\Big)
  + G_\Lambda\Big(\frac{\alpha_2}{2}\Big) \Big), \ \ |D_\alpha|=1 
  \end{equation}
where $\Lambda$ is the lattice generated by $\alpha$ and
\begin{equation}
  \label{b}
   b=  \frac{2 \gamma_{12} r_1^2r_2^2}
       { \gamma_{11} r_1^4 + \gamma_{22} r_2^4} 
   =  \frac{2 \gamma_{12} \omega_1\omega_2}
       { \gamma_{11} \omega_1^2 + \gamma_{22} \omega_2^2}.
\end{equation}
\end{lemma}

\begin{proof}
Given a disc $B(\xi_j, r_j)$ one finds
\begin{align}
  \label{I-B} \nonumber
  I_\Lambda(B(\xi_j,r_j))(\zeta) &= 
  \left \{ \begin{array}{ll}
    - \frac{|\zeta - \xi_j|^2}{4} +\frac{r_j^2}{4} - \frac{r_j^2}{2} \log r_j, 
        & \mbox{if} \ \ |\zeta - \xi_j|\in [0, r_j], \\
   -\frac{r_j^2}{2} \log |\zeta -\xi_j|, & \mbox{if} \ \ |\zeta-\xi_j| > r_j,
  \end{array} \right.  \\
  & \qquad - \frac{r_j^2}{2} \log \frac{2\pi}{\sqrt{|\Lambda|}} 
   + \frac{1}{4|\Lambda|}  \Big (
  \pi r_j^2 |\zeta - \xi_j|^2 + \frac{\pi r_j^4}{2} \Big )
  + \pi r_j^2 H_\Lambda(\zeta - \xi_j) 
\end{align}
by \eqref{Green} and the mean value property of the harmonic function
$H_\Lambda$.
Then
\begin{align}
  \label{Bj-Bj}
  \int_{B(\xi_j, r_j)} \int_{B(\xi_j,r_j)} G_\Lambda(\zeta - \chi)
    \, d\chi d\zeta 
   &= \int_{B(\xi_j, r_j)} I_\Lambda(B(\xi_j,r_j))(\zeta) \, d \zeta \\
   &= \pi^2 r_j^4 H_\Lambda(0) + \frac{\pi r_j^4}{8} - \frac{\pi r_j^4}{2}
   \log \frac{2\pi r_j}{ \sqrt{|\Lambda|}}
    + \frac{\pi^2 r_j^6}{4|\Lambda|}.  \nonumber 
\end{align}
When $j\ne k$,
\begin{align}
  \label{Bj-Bk}
  \int_{B(\xi_k, r_k)} \int_{B(\xi_j,r_j)} G_\Lambda(\zeta - \chi)
    \, d\chi d\zeta 
   &= \int_{B(\xi_k, r_k)} I_\Lambda(B(\xi_j,r_j))(\zeta) \, d \zeta \\
   &= \pi^2 r_j^2 r_k^2 G_\Lambda(\xi_j-\xi_k)
   + \frac{\pi^2(r_j^2 r_k^4 + r_j^4 r_k^2)}{8 |\Lambda|}. \nonumber
\end{align}

Since only the role played by the lattice basis $\alpha$ is of interest,
let
\begin{equation}
  \label{c}
   c_{jj} = \frac{\pi r_j^4}{8} - \frac{\pi r_j^4}{2}
   \log \frac{2\pi r_j}{ \sqrt{|\Lambda|}}
   + \frac{\pi^2 r_j^6}{4|\Lambda|}, \ \
   c_{jk} =  \frac{\pi^2(r_j^2 r_k^4 + r_j^4 r_k^2)}{8 |\Lambda|}, \ j\ne k
  \end{equation}
which are independent of $\alpha$ when $|\Lambda|=1$. Then
\begin{align}
  \int_{B(\xi_j, r_j)} \int_{B(\xi_j,r_j)} G_\Lambda(\zeta - \chi)
  \, d\chi d\zeta
  &= \pi^2 r_j^4 H_\Lambda(0) + c_{jj} \\
  \int_{B(\xi_k, r_k)} \int_{B(\xi_j,r_j)} G_\Lambda(\zeta - \chi)
  \, d\chi d\zeta
  &= \pi^2 r_j^2 r_k^2 G_\Lambda(\xi_j-\xi_k) + c_{jk}.
  \end{align}
Similarly,
\begin{align}
  \int_{B(\xi_j', r_j)} \int_{B(\xi_j',r_j)} G_\Lambda(\zeta - \chi)
  \, d\chi d\zeta
  &= \pi^2 r_j^4 H_\Lambda(0) + c_{jj} \\
  \int_{B(\xi_k', r_k)} \int_{B(\xi_j',r_j)} G_\Lambda(\zeta - \chi)
  \, d\chi d\zeta
  &= \pi^2 r_j^2 r_k^2 G_\Lambda(\xi_j'-\xi_k') + c_{jk}, \ \ j \ne k \\
  \int_{B(\xi_k, r_k)} \int_{B(\xi_j',r_j)} G_\Lambda(\zeta - \chi)
  \, d\chi d\zeta
  &= \pi^2 r_j^2 r_k^2 G_\Lambda(\xi_j-\xi_k') + c_{jk}',
  \ \ j,k=1,2  \label{Bk-Bj'}
\end{align}
where
\begin{equation}
  \label{c'}
  c_{jk}' = \frac{\pi^2(r_j^2 r_k^4 + r_j^4 r_k^2)}{8 |\Lambda|}, \ \
  j,k =1,2.
  \end{equation}
Note that in \eqref{Bk-Bj'} and \eqref{c'} $j$ may be equal to $k$.

To complete the computation, note that
\begin{align*}
  \int_{\Omega_{\alpha,1}}\int_{\Omega_{\alpha,1}} G_\Lambda(\zeta -\chi) \, d\chi d\zeta
  &=  \int_{B(\xi_1,r_1) \cup B(\xi_1',r_1)}  \int_{B(\xi_1,r_1) \cup B(\xi_1',r_1)}
  G_\Lambda(\zeta -\chi) \, d\chi d\zeta \\
  &= \int_{B(\xi_1,r_1)}  \int_{B(\xi_1,r_1)}
  G_\Lambda(\zeta -\chi) \, d\chi d\zeta
  + \int_{B(\xi_1',r_1)}  \int_{B(\xi_1',r_1)}
  G_\Lambda(\zeta -\chi) \, d\chi d\zeta \\
  & \qquad + \int_{B(\xi_1,r_1)}  \int_{B(\xi_1',r_1)}
  G_\Lambda(\zeta -\chi) \, d\chi d\zeta
  + \int_{B(\xi_1',r_1)}  \int_{B(\xi_1,r_1)}
  G_\Lambda(\zeta -\chi) \, d\chi d\zeta \\
  &= 2(\pi^2 r_1^4 H_\Lambda(0) + c_{11})
  + 2(\pi^2 r_1^4 G_\Lambda(\xi_1-\xi_1') + c_{11}') \\
 \int_{\Omega_{\alpha,2}}\int_{\Omega_{\alpha,2}} G_\Lambda(\zeta -\chi) \, d\chi d\zeta
  &=  2(\pi^2 r_2^4 H_\Lambda(0) + c_{22})
 + 2(\pi^2 r_2^4 G_\Lambda(\xi_2-\xi_2') + c_{22}') \\
 \int_{\Omega_{\alpha,1}}\int_{\Omega_{\alpha,2}} G_\Lambda(\zeta -\chi) \, d\chi d\zeta
 &= \pi^2 r_1^2 r_2^2 G_\Lambda(\xi_1-\xi_2) + c_{12}
   + \pi^2 r_1^2 r_2^2 G_\Lambda(\xi_1'-\xi_2') + c_{12} \\
  &\qquad + \pi^2 r_1^2 r_2^2 G(\xi_1-\xi_2') + c'_{12} + \pi^2 r_1^2 r_2^2 G(\xi_1'-\xi_2) + c'_{12}
\end{align*}
In accordance to \eqref{Omegaalpha1} and \eqref{Omegaalpha2} choose
\begin{align}
 \xi_1 &= \frac{3}{4} \alpha_1 + \frac{1}{4} \alpha_2, \ \xi_1'=\frac{1}{4} \alpha_1 + \frac{3}{4} \alpha_2 \\
 \xi_2 &= \frac{1}{4} \alpha_1 + \frac{1}{4} \alpha_2, \ \xi_1'=\frac{1}{4} \alpha_1 + \frac{3}{4} \alpha_2
\end{align}
to derive
\begin{align*}
  {\cal F}(\alpha) &= \sum_{j,k=1,2} \frac{\gamma_{jk}}{2}
  \int_{\Omega_{\alpha,j}}\int_{\Omega_{\alpha,k}}
  G_\Lambda(\zeta -\chi) \, d\chi d\zeta \\
  &= \gamma_{11} \Big (  \pi^2 r_1^4 H_\Lambda(0) + 
  + \pi^2 r_1^4 G_\Lambda(\frac{\alpha_1-\alpha_2}{2}) + c_{11} + c_{11}' \Big ) \\
  & \qquad + \gamma_{22} \Big (  \pi^2 r_2^4 H_\Lambda(0) + 
  + \pi^2 r_2^4 G_\Lambda(\frac{\alpha_1+\alpha_2}{2}) + c_{22} + c_{22}' \Big ) \\
  & \qquad + 2 \gamma_{12} \Big (  \pi^2 r_1^2r_2^2
  \Big (G_\Lambda(\frac{\alpha_1}{2}) + G_\Lambda(\frac{\alpha_2}{2}) \Big )
  + c_{12} + c_{12}' \Big ) \\
  &= (\gamma_{11} \pi^2 r_1^4 + \gamma_{22} \pi^2 r_2^4)
  \Big ( H_\Lambda(0) + G_\Lambda(\frac{\alpha_1+\alpha_2}{2}) \Big )
  + 2 \gamma_{12} \pi^2 r_1^2 r_2^2 \Big ( G_\Lambda(\frac{\alpha_1}{2})
  + G_\Lambda(\frac{\alpha_2}{2}) \Big ) \\
  & \qquad + \gamma_{11}(c_{11} + c_{11}') + \gamma_{22} (c_{22}+c_{22}')
  + 2 \gamma_{12} (c_{12} + c_{12}') \\
  &= (\gamma_{11} \pi^2 r_1^4 + \gamma_{22} \pi^2 r_2^4)
  \Big [ H_\Lambda(0) + G_\Lambda(\frac{\alpha_1+\alpha_2}{2})   +
  \frac{2 \gamma_{12} r_1^2r_2^2}
       { \gamma_{11} r_1^4 + \gamma_{22} r_2^4} 
        \Big ( G_\Lambda(\frac{\alpha_1}{2})
        + G_\Lambda(\frac{\alpha_2}{2}) \Big ) \Big ] \\
        & \qquad +  \gamma_{11}(c_{11} + c_{11}') + \gamma_{22} (c_{22}+c_{22}')
  + 2 \gamma_{12} (c_{12} + c_{12}').
\end{align*}
Here 
$G_\Lambda(\frac{\alpha_1+\alpha_2}{2}) = G_\Lambda(\frac{\alpha_1-\alpha_2}{2})$
follows from the $\Lambda$-periodicity of $G_\Lambda$.
\end{proof}

Calculations based on \eqref{H} show
\begin{align}
  \label{H0} 
  H_\Lambda(0) &= - \frac{1}{2\pi} \log \Big | \sqrt{\im \tau}
  \e (\frac{\tau}{12}) \prod_{n=1}^\infty (1- \e(n \tau))^2 \Big | \\
  \label{Galpha1alpha2}
  G_\Lambda(\frac{\alpha_1}{2}+\frac{\alpha_2}{2}) &= 
  -\frac{1}{2\pi} \log \Big | \e(-\frac{\tau}{24}) \prod_{n=1}^\infty
  \Big (1 + \e \big ((n-\frac{1}{2}) \tau) \Big )^2  \Big |
  \\ \label{Galpha1}
  G_\Lambda(\frac{\alpha_1}{2}) &= - \frac{1}{2\pi}
  \log \Big | 2  \e(\frac{\tau}{12}) \prod_{n=1}^\infty
  (1+ \e(n\tau))^2  \Big | \\
  \label{Galpha2}
  G_\Lambda(\frac{\alpha_2}{2}) &= 
  -\frac{1}{2\pi} \log \Big |  \e(-\frac{\tau}{24}) \prod_{n=1}^\infty
  \Big (1- \e \big ((n-\frac{1}{2})\tau \big ) \Big)^2 \Big |. 
\end{align}
To derive \eqref{H0} we have used $\frac{1}{|\alpha_1|} = \sqrt{\im \tau}$,
which follows from $1=|D_\Lambda|= \im (\overline{\alpha_1} \alpha_2)$. 

Regarding the four infinite products in \eqref{H0} through \eqref{Galpha2},
one has the following formulas.

\begin{lemma}
  \label{l-formulae}
  \begin{align*}
  \prod_{n=1}^\infty (1- \e(n \tau)) 
   \prod_{n=1}^\infty
  \Big (1 + \e \big ((n-\frac{1}{2}) \tau) \Big )
   = \prod_{n=1}^\infty \Big ( 1- \e(n \frac{\tau +1}{2}) \Big )
   \\ 
   \prod_{n=1}^\infty
   \Big (1 + \e \big ((n-\frac{1}{2}) \tau) \Big )
    \prod_{n=1}^\infty (1+ \e(n \tau)) 
   \prod_{n=1}^\infty
   \Big (1 - \e \big ((n-\frac{1}{2}) \tau) \Big )
   = 1.
   \end{align*}
  \end{lemma} 

\begin{proof}
  To prove the first formula, rewrite and rearrange the terms as follows.
  \begin{align*}
     & \prod_{n=1}^\infty (1- \e(n \tau)) 
   \prod_{n=1}^\infty
   \Big (1 + \e \big ((n-\frac{1}{2}) \tau) \Big ) \\
   & = \prod_{n=1}^\infty \Big ( 1- \e(2n \frac{\tau+1}{2}) \Big )
   \prod_{n=1}^\infty \Big ( 1- \e((2n-1) \frac{\tau+1}{2}) \Big ) \\
   & = \prod_{n=1}^\infty \Big ( 1- \e(
   n \frac{\tau+1}{2}) \Big ).
    \end{align*}
  For the second formula, consider
  \begin{align*}
    & \prod_{n=1}^\infty (1-\e(n\tau))
     \prod_{n=1}^\infty
   \Big (1 + \e \big ((n-\frac{1}{2}) \tau) \Big )
    \prod_{n=1}^\infty (1+ \e(n \tau)) 
   \prod_{n=1}^\infty
   \Big (1 - \e \big ((n-\frac{1}{2}) \tau) \Big ) \\
   & =  \prod_{n=1}^\infty (1-\e(2n\tau))
    \prod_{n=1}^\infty
    \Big (1 - \e \big ((2n-1) \tau) \Big ) \\
   & = \prod_{n=1}^\infty (1-\e(n\tau)).
  \end{align*}
  The second formula follows after one divides out
  $\prod_{n=1}^\infty (1-\e(n\tau))$.
  \end{proof}

These identities will allow us to further simplify
$\widetilde{\cal F}(\alpha)$. Let
\begin{equation}
  \mathbb{H} = \{ z \in \mathbb{C}: \ \im z >0 \}
  \label{mathH}
\end{equation}
be the upper half of the complex plane. Define
\begin{equation}
  f_b (z) = b \log \big | \im \big (z \big ) \eta\big (z \big) \big|
  + (1-b) \log \big | \im \big(\frac{z+1}{2}\big)
  \eta \big(\frac{z+1}{2}\big) \big |, \ z \in \mathbb{H}, 
  \label{fb}
\end{equation}
where
\begin{equation}
  \label{eta}
  \eta(z) = \e \Big(\frac{z}{6} \Big) \prod_{n=1}^\infty\big(1- \e(nz)\big)^4.
\end{equation}
One often writes $f_b$ as 
\begin{equation}
  \label{fb-2}
  f_b(z) = b f_1(z) + (1-b) f_0(z),
  \end{equation}
where
\begin{equation}
  \label{f1f0}
  f_1(z) = \log |\im(z) \eta(z)|, \ \ f_0(z) = \log
  |\im(\frac{z+1}{2}) \eta (\frac{z+1}{2})|. 
\end{equation}
  
\begin{lemma}
  \label{l-Ff}
  Minimizing $\widetilde{\cal F}(\alpha)$ with respect to
  $\alpha$ of unit cell area is equivalent to maximizing
  $f_b(z)$ with respect to $z$ in $\mathbb{H}$. 
  \end{lemma}

\begin{proof}
By the first formula in Lemma \ref{l-formulae},
\begin{align}
  H_\Lambda(0) + G_\Lambda(\frac{\alpha_1 + \alpha_2}{2}) 
  & = - \frac{1}{2\pi} \log \Big | \sqrt{\im \tau}
  \e(\frac{\tau}{24}) \prod_{n=1}^\infty
  \Big ( 1- \e(n \frac{\tau +1}{2}) \Big )^2
  \Big | \nonumber \\ \label{T1}
  & = -\frac{1}{4\pi} \log | \im (\frac{\tau+1}{2})
  \eta(\frac{\tau+1}{2}) \Big | - \frac{1}{4\pi} \log 2.
\end{align}
Using both formulas in Lemma \ref{l-formulae}, one deduces 
\begin{align}
  G_\Lambda(\frac{\alpha_1}{2}) + G_\Lambda(\frac{\alpha_2}{2}) 
  & = - \frac{1}{2\pi} \log \Big | 2 
  \e(\frac{\tau}{24}) \prod_{n=1}^\infty
  ( 1 + \e(n \tau) )^2
   \prod_{n=1}^\infty
  \Big ( 1 - \e((n-\frac{1}{2}) \tau) \Big )^2
  \Big | \nonumber \\ \nonumber
  & = - \frac{1}{2\pi} \log \Big | 2 
  \e(\frac{\tau}{24}) \prod_{n=1}^\infty
  ( 1 - \e(n \tau) )^2
   \Big / \prod_{n=1}^\infty
  \Big ( 1 - \e(n \frac{\tau+1}{2}) \Big )^2
  \Big | \\ \label{T2}
  &= -\frac{1}{4\pi} \Big ( \log |\im (\tau) \eta (\tau)|
  -  \log \Big |\im (\frac{\tau+1}{2}) \eta (\frac{\tau+1}{2}) \Big | \Big  )
  -\frac{1}{4\pi} \log 2.
\end{align}
By \eqref{T1} and \eqref{T2}, $\widetilde{\cal F}(\Lambda)$
of \eqref{Ftilde} is reduced to
\begin{equation}
  \widetilde{\cal F}(\Lambda) = -\frac{1}{4\pi}
  \Big (
  b  \log |\im (\tau) \eta (\tau)| +
  (1-b)  \log \Big |\im (\frac{\tau+1}{2}) \eta (\frac{\tau+1}{2}) \Big |
  \Big ) - \frac{1+b}{4\pi} \log 2,
  \end{equation}
from which the lemma follows.
\end{proof}

\section{Duality property of $f_b$}
\setcounter{equation}{0}

The function $\eta$ in the definition of $f_b$ satisfies two functional
equations.
\begin{lemma}
  \label{l-func-eq}
  For all $z \in \mathbb{H}$, 
  \begin{align}
    \label{inv-1} \eta(z+1) & =e^{\frac{2\pi i}{6}}\eta(z), \\
    \label{inv-2}  \eta\Big (-\frac{1}{z} \Big ) & = -z^2 \eta(z).
    \end{align}
  \end{lemma}
\begin{proof}
  The function $\eta$ in \eqref{eta} is the fourth
  power of the Dedekind eta function which is
\[ \eta_D(z) =  \e\Big (\frac{z}{24} \Big)
   \prod_{n=1}^\infty\big(1- e^{2\pi i n z}\big) \]
so
   \[ \eta(z) = \eta_{D}^4(z), \ \ z \in \mathbb{H}. \]
   For the Dedekind eta function, it is known \cite[Chapter 2]{apostol-2} that
   \begin{align}
    \eta_D(z+1) & = e^{\frac{2\pi i }{24}} \eta_D (z) \\
    \eta_D \Big (-\frac{1}{z} \Big) &= \sqrt{-i z} \, \eta_D(z)
   \end{align}
   where $\sqrt{\cdot }$ stands for the principal branch of squaure root.
\end{proof}

These functional equations lead to invariance properties.
\begin{lemma}
  \label{l-invariance}
  \begin{enumerate}
  \item $|\im (z) \eta(z) |$, and consequently $f_1(z)$,
    are invariant under the transforms
   \[ z \rightarrow z+1, \ \ \mbox{and} \ \ z \rightarrow -\frac{1}{z}.\]
 \item $| \im \big(\frac{z+1}{2}\big) \eta \big(\frac{z+1}{2}\big) |$,
   and consequently $f_0(z)$, are invariant under the transforms
   \[ z \rightarrow z+2, \ \ \mbox{and} \ \ z \rightarrow -\frac{1}{z}.\] 
    \end{enumerate}
  \end{lemma}

\begin{proof}
  The invariance of 
  $ | \im \big (z \big ) \eta\big (z \big) |$ under $z \rightarrow z+1$ and
  the invariance of 
$| \im \big(\frac{z+1}{2}\big) \eta \big(\frac{z+1}{2}\big) |$ 
under $z \rightarrow z+2$ follow from \eqref{inv-1}. 
By \eqref{inv-2} it is easy to see that
\begin{equation}
  \label{inv-1-2}
  | \im(-\frac{1}{z}) \eta(-\frac{1}{z}) | = |\im (z) \eta(z) |,
\end{equation}
so $| \im \big (z \big ) \eta\big (z \big) |$ is invariant under
$z \rightarrow - \frac{1}{z}$.

The invariance of $| \im \big (z \big ) \eta\big (z \big) |$ under
$z \rightarrow z+1$ implies its invariance uder
$z \rightarrow z+k$ for any integer $k$. Now one deduces
\begin{align}
\big | \im \big(\frac{(-\frac{1}{z})+1}{2}\big)
\eta \big(\frac{(-\frac{1}{z})+1}{2}\big) \big | 
&= \big | \im \big(\frac{z-1}{2z}\big)
\eta \big(\frac{z-1}{2z}\big) \big | \nonumber \\
&= \big | \im \big(\frac{-z-1}{2z}\big)
\eta \big(\frac{-z-1}{2z}\big) \big | \nonumber \\
&= \big | \im \big(\frac{2z}{z+1}\big)
\eta \big(\frac{2z}{z+1}\big) \big | \nonumber \\
&= \big | \im \big(\frac{-2}{z+1}\big)
\eta \big(\frac{-2}{z+1}\big) \big | \nonumber \\
&= \big | \im \big(\frac{z+1}{2}\big)
\eta \big(\frac{z+1}{2}\big) \big | \label{2-2}
  \end{align}
by applying the invariance of
$| \im \big (z \big ) \eta\big (z \big) |$ under
$z \rightarrow z -1$, $z \rightarrow -\frac{1}{z}$,
$z \rightarrow z-2$, and $z \rightarrow -\frac{1}{z}$ successively. 
This proves the invariance of
$| \im \big(\frac{z+1}{2}\big) \eta \big(\frac{z+1}{2}\big) |$
under $z \rightarrow -\frac{1}{z}$.
\end{proof}

There is another invariance that is not a linear fractional transform:
the reflection about the imaginary axis.

\begin{lemma}
  \label{l-invariance-reflection}
   Both $ | \im \big (z \big ) \eta\big (z \big) |$  and
   $| \im \big(\frac{z+1}{2}\big) \eta \big(\frac{z+1}{2}\big) |$,
   and consequently $f_1(z)$ and $f_0(z)$, are invariant under
   $z \rightarrow - \bar{z}$.
  \end{lemma}
\begin{proof}
  These follow easily from the infinite product definition \eqref{eta}
  of $\eta$. 
  \end{proof}

The transforms
$z \rightarrow z+1$ and $z \rightarrow -\frac{1}{z}$ generate
the modular group $\Gamma$, 
\[ \Gamma = \left \{ z \rightarrow \frac{az + b}{cz +d}: a,b,c,d \in \mathbb{Z},
\ ad - bc =1 \right \}. \]
And $\Gamma$ has
\begin{equation}
  \label{gamma-fund}
   F_\Gamma=\{ z \in \mathbb{H}: \ |z| >1, \ -1/2 < \re z < 1/2 \}
  \end{equation}
as a fundamental region. It means that every orbit under this
group has one element in
$\overline{F_\Gamma}_{\mathbb{H}}$, the closure of $F_\Gamma$ in $\mathbb{H}$,
and no two points in $F_\Gamma$ belong to the same orbit \cite{apostol-2}.

The transforms $z \rightarrow z+2$ and $z \rightarrow -\frac{1}{z}$
generate a subgroup $\Gamma'$ of $\Gamma$,
\begin{align}
  \label{Gamma-2}
  \Gamma' & = \Big \{ z \rightarrow \frac{az +b}{cz +d} \in \Gamma:
  \ a \equiv d \equiv 1 \mod 2 \ \ \ \mbox{and} \ \ 
  \ b \equiv c \equiv 0 \mod 2, \nonumber \\  & \qquad \mbox{or} \
 \ a \equiv d \equiv 0 \mod 2 \ \ \ \mbox{and} \ \ 
  \ b \equiv c \equiv 1 \mod 2
 \Big  \}. 
  \end{align}
It is known in number theory that this group has
\begin{equation}
    F_{\Gamma'}=\{ z \in \mathbb{H}: \ |z| >1, \ -1 < \re z < 1 \}
  \end{equation}
as a fundamental region \cite{evans_ronald}.

Denote by ${\cal G}$ the group of diffeomorphisms of
$\mathbb{H}$ generated by
\[ z \rightarrow z+2, \ \ z \rightarrow - \frac{1}{z},
\ \ z \rightarrow - \bar{z}. \]
Note that $\Gamma'$ is a subgroup of ${\cal G}$ but $\Gamma$ is not
a subgroup of ${\cal G}$.

With the group ${\cal G}$, maximizing $f_b$ need
not be carried out in $\mathbb{H}$, but in a smaller set which contains at least one element from each orbit of ${\cal G}$. Let
\begin{equation}
    \label{W}
    W = \{ z \in \mathbb{H}: \ 0 < \re z <1, \ |z| >1 \}
\end{equation}
and
\begin{equation}
    \label{W-closure}
    \overline{W}_{\mathbb{H}}
    = \{ z \in \mathbb{H}: \ 0 \leq \re z \leq 1, \ |z| \geq 1 \};
\end{equation}
see Figure \ref{f-W}.
Note that $\overline{W}_{\mathbb{H}}$ is the closure of $W$ in $\mathbb{H}$
so $1 \not \in \overline{W}_{\mathbb{H}}$.
\begin{figure}
\centering
 \includegraphics[scale=0.13]{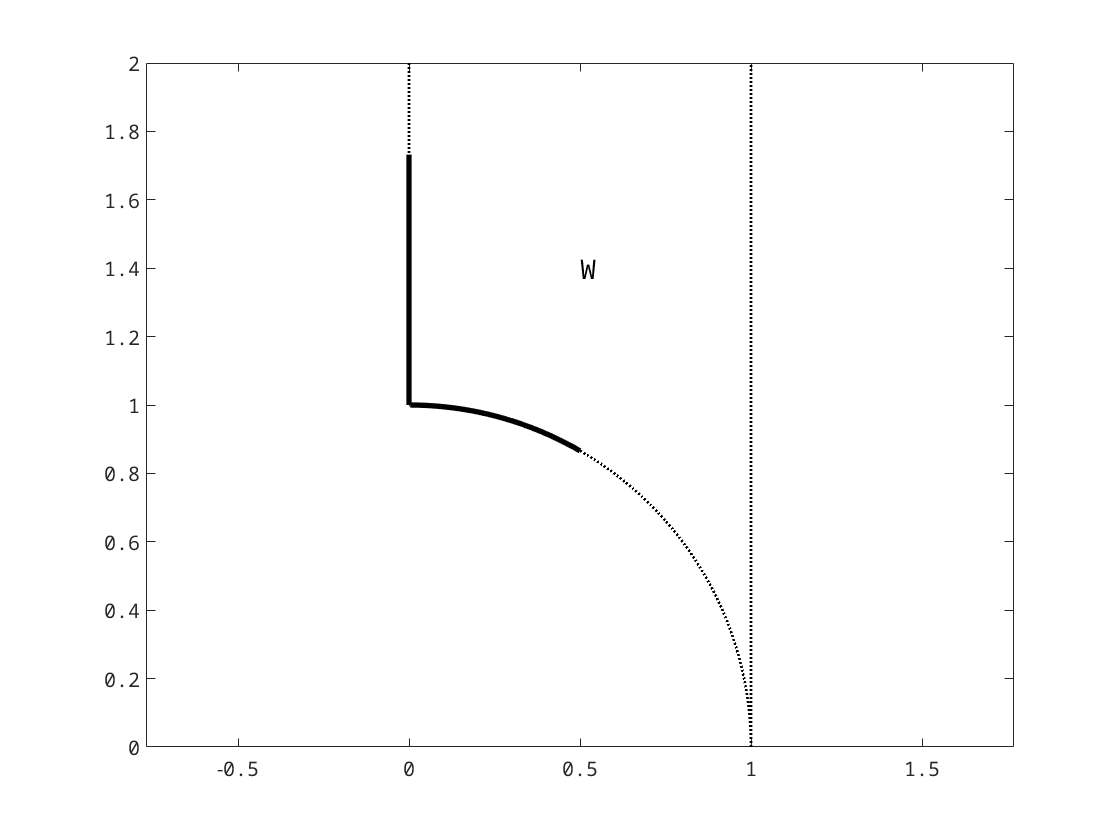} 
 \includegraphics[scale=0.13]{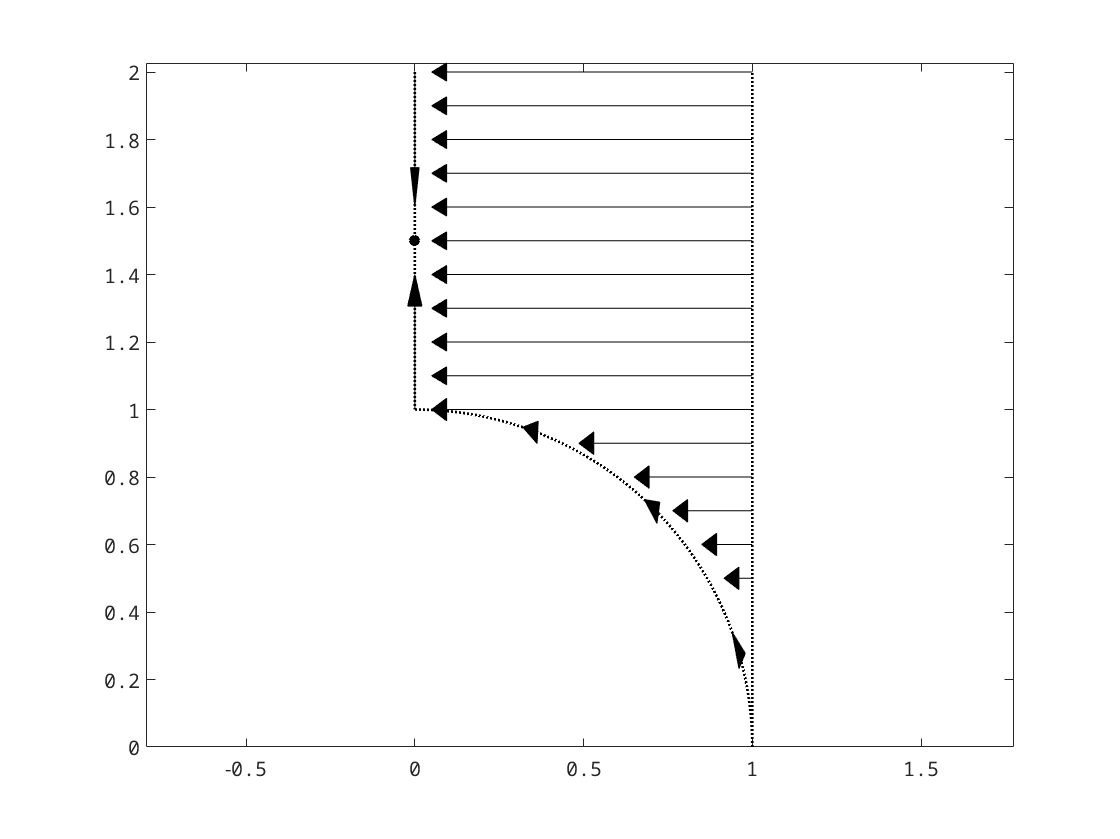}
 \caption{The left plot shows the set $W$. In $\overline{W}_{\mathbb{H}}$,
   $f_b$ attain the maximum at a point either on the thick line segment
   or on the thick arc. In the right plot, with  $b \in (0,B)$,
   $f_b$ increases in the directions of the arrows. The dot on the
 imaginary axis is $q_b i$.}
\label{f-W}
\end{figure}

\begin{lemma}
  \label{l-inv-group}
  \begin{enumerate}
  \item $| \im \big (z \big ) \eta\big (z \big)|$ and $f_1(z)$ are
    invariant under the group $\Gamma$ and the transform
    $z \rightarrow -\bar{z}$.
  \item $| \im \big(\frac{z+1}{2}\big) \eta \big(\frac{z+1}{2}\big) |$,
    $f_0(z)$, and $f_b(z)$, for $b \in \mathbb{R}$, are invariant
    under the groups $\Gamma'$ and ${\cal G}$.
  \item As the group ${\cal G}$ acts on $\mathbb{H}$, each orbit of ${\cal G}$
    has at least one element in $\overline{W}_{\mathbb{H}}$.
 \end{enumerate}
\end{lemma}

\begin{proof}
  Part 1 and part 2 follow from Lemmas \ref{l-invariance} and
  \ref{l-invariance-reflection}. Part 3 follows from
  $F_{\Gamma'}$ being the fundamental region of $\Gamma'$ and the
  transform $z \rightarrow - \bar{z} \in {\cal G}$.
\end{proof}

It is instructive to understand transforms from the view point of bases. Let $\alpha=(\alpha_1, \alpha_2)$
and $\alpha'=(\alpha'_1, \alpha'_2)$ be two bases of unit cell area that define lattices $\Lambda$ and $\Lambda'$, and two species
periodic assemblies $(\Omega_{\alpha,1}, \Omega_{\alpha,2})$ and $(\Omega_{\alpha',1}, \Omega_{\alpha',2})$.
Set $z = \frac{\alpha_2}{\alpha_1}$ and $z'=\frac{\alpha'_2}{\alpha'_1}$.

If $\alpha$ and $\alpha'$ are related by the transform $z \rightarrow z'=z+1$, then 
\[ \frac{\alpha_2'}{\alpha_1'} = \frac{\alpha_1+\alpha_2 }{\alpha_1}  \]
and consequently there exists $ \kappa \in \mathbb{C}$ such that 
$\alpha'_1 = \kappa \alpha_1$ and  $\alpha'_2 = \kappa (\alpha_1+\alpha_2)$. 
Since both bases have unit cell area,
\[ 1= \im(\overline{\alpha'_1} \alpha'_2) = |\kappa|^2 \im(\overline{\alpha_1} (\alpha_1+\alpha_2))
= |\kappa|^2 \im (\overline{\alpha_1} \alpha_2) =|\kappa|^2. \]
So there exists $\theta \in [0, 2\pi)$ such that $\kappa = e^{\theta i}$  and 
$\alpha'= e^{\theta i} (\alpha_1, \alpha_1+\alpha_2)$. Let $\alpha''=e^{\theta i} \alpha=e^{\theta i}(\alpha_1,\alpha_2)$. Then
$\alpha'$ and $\alpha''$ generate the same lattice $\Lambda'$, and consequently
$\Lambda$ and $\Lambda'$ are isometric in the sense that a parallelogram cell $D_\alpha$
of $\Lambda$ is just a rotation of a parallelogram cell $D_{\alpha''}$ of $\Lambda'$. However 
the assemblies $(\Omega_{\alpha,1}, \Omega_{\alpha,2})$ and $(\Omega_{\alpha',1}, \Omega_{\alpha',2})$ are usually quite
different since in general $ {\cal J}_\Lambda(\Omega_{\alpha,1}, \Omega_{\alpha,2}) \ne {\cal J}_{\Lambda'}(\Omega_{\alpha',1}, \Omega_{\alpha',2})$.

The story changes if $\alpha$ and $\alpha'$ are related by the transform $z \rightarrow z'=z+2$. This time not only
$\Lambda$ and $\Lambda'$ are isometric, $ {\cal J}_\Lambda(\Omega_{\alpha,1}, \Omega_{\alpha,2}) = {\cal J}_{\Lambda'}(\Omega_{\alpha',1}, \Omega_{\alpha',2})$ as well by Lemma \ref{l-inv-group}.2.

If $\alpha$ and $\alpha'$ are related by $z \rightarrow z'=-\frac{1}{z}$, one can show that 
$(\alpha'_1,\alpha'_2)=e^{\theta i}(-\alpha_2,\alpha_1)$, $\theta \in [0,2\pi)$. Then $\Lambda$ and $\Lambda'$ are 
isometric since $D_\alpha$ and $D_{\alpha'}$ differ by a translation and rotation. Moreover
$ {\cal J}_\Lambda(\Omega_{\alpha,1}, \Omega_{\alpha,2}) = {\cal J}_{\Lambda'}(\Omega_{\alpha',1}, \Omega_{\alpha',2})$ by Lemma \ref{l-inv-group}.2.

Finally if $\alpha$ and $\alpha'$ are related by $z \rightarrow z'=-\bar{z}$, then 
$(\alpha'_1,\alpha'_2)=e^{\theta i}(\overline{\alpha_1},-\overline{\alpha_2})$, $\theta \in [0,2\pi)$.  Therefore
$D_\alpha$ and $D_{\alpha'}$ differ by a mirror reflection, a translation and a rotation, so $\Lambda$ and
$\Lambda'$ are again isometric. By Lemma \ref{l-inv-group}.2, 
$ {\cal J}_\Lambda(\Omega_{\alpha,1}, \Omega_{\alpha,2}) = {\cal J}_{\Lambda'}(\Omega_{\alpha',1}, \Omega_{\alpha',2})$.

In summary if $z$ and $z'$ are related by an element $g \in {\cal G}$, i.e. $z'=g z$, then $\Lambda$ and $\Lambda'$ are
isometric and $ {\cal J}_\Lambda(\Omega_{\alpha,1}, \Omega_{\alpha,2}) = {\cal J}_{\Lambda'}(\Omega_{\alpha',1}, \Omega_{\alpha',2})$.  For isometric $\Lambda$ and $\Lambda'$, if $\Lambda$ is a rectangular lattice, then $\Lambda'$ is also a rectangular lattice of the same ratio of longer side and shorter side; if $\Lambda$ is a rhombic lattice, then $\Lambda'$ is also a rhombic lattice of the same acute angle. Therefore to prove Theorem \ref{t-main}, it suffices to find every minimum of $f_b$ in $\overline{W}_{\mathbb{H}}$
and identify its associated  lattice. 

In the exceptional case $b=1$, since $f_1$ is invariant under $\Gamma$,
it suffices to maximize $f_1$ in a smaller set
\begin{equation}
  \label{U-closure}
  \overline{U}_{\mathbb{H}} = \{ z \in \mathbb{H}: \
  0 \leq \re z \leq 1/2, \ |z| \geq 1 \}
\end{equation}
which is the closure of
  \begin{equation}
    U = \Big \{ z \in \mathbb{C}: \ |z| >1, \ 0< \re z < \frac{1}{2} \Big \}.
    \label{U}
  \end{equation}
This fact was used critically in
\cite{chenoshita-2}, but it is not valid if $b \ne 1$.
The approach of this paper works for all $b \in [0,1]$,
 giving a different proof for the $b=1$ case as well. 

One of the most important properties of $f_b$ is the following 
duality relation between $f_b$ and $f_{1-b}$.

\begin{lemma}
\label{l-dual}
Under the transform $z \rightarrow w = \frac{z-1}{z+1}$ of $\mathbb{H}$, 
\[ f_1(z) = f_0 (w) \ \ \mbox{and} \ \ f_0(z) = f_1(w), \ \ z \in \mathbb{H} \ \ \mbox{and} \ \ 
 w=\frac{z-1}{z+1} \in \mathbb{H}. \]
Consequently, for all $b \in \mathbb{R}$,
\[ f_b(w) = f_{1-b} (z), \ \  z \in \mathbb{H} \ \ \mbox{and} \ \ 
 w=\frac{z-1}{z+1} \in \mathbb{H}. \]

 More generally, if $h: z' \rightarrow w'= \frac{z'-1}{z'+1}$ and
 $g_1: z \rightarrow z'$, $g_2: w' \rightarrow w$ are transforms in ${\cal G}$,
 then
 \[ f_b(w) = f_{1-b} (z), \ \  z \in \mathbb{H} \ \ \mbox{and} \ \ 
 w=g_2 \circ h \circ g_1 (z) \in \mathbb{H}.\]
\end{lemma} 

\begin{proof}
The transform $z \rightarrow w'=\frac{z}{z+1}$ is in $\Gamma$, so 
\[ f_1(z) = f_1(w') \]
by the invariance of $f_1$ under $\Gamma$. On the other hand substitution shows
\[ f_0(z) = f_0(\frac{w'}{-w'+1})=\log | \im(\frac{1}{-2w'+2}) \eta (\frac{1}{-2w'+2})|
= f_1(\frac{1}{-2w'+2}). \]
Now apply another transform $w' \rightarrow w = 2 w'-1$ which is not in $\Gamma$ to 
find 
\[ f_1(z) = f_1(\frac{w+1}{2})=f_0(w) \]
and 
\[ f_0(z) = f_1 (\frac{1}{-w+1}) = f_1(w)\]
where the last equation follows from 
the invariance of $f_1$ under $ w \rightarrow \frac{1}{-w+1} \in \Gamma$.

The composition of the two transforms is
$z \rightarrow w'=\frac{z}{z+1} \rightarrow w=2w'-1= \frac{z-1}{z+1}$. 
\end{proof}

Although $b$ is supposed to be in $[0,1]$ throughout this paper,
some properties of $f_b$, like Lemma \ref{l-dual},
hold for $b \in \mathbb{R}$. If this is the case we state so explicitly.

Let us write $z=x+yi$ henceforth, and set
\[ X_b(z) = \frac{\partial f_b(z)}{\partial x} = b X_1(z) + (1-b) X_0(z), \ \
Y_b(z) = \frac{\partial f_b(z)}{\partial y}= b Y_1(z) + (1-b) Y_0(z), \]
where
\begin{align}
  \label{X1Y1}
  X_1(z) & = \frac{\partial}{\partial x} \log |\im(z) \eta(z)|, \ \
  Y_1(z) = \frac{\partial}{\partial y} \log |\im(z) \eta(z)| \\
  X_0(z) & = \frac{\partial}{\partial x} \log |\im(\frac{z+1}{2})
  \eta(\frac{z+1}{2})|, \ \
  Y_0(z) = \frac{\partial}{\partial y} \log |\im(\frac{z+1}{2})
  \eta(\frac{z+1}{2})|.
  \end{align}
These functions can be written as the following series.
  \begin{align}
    \label{X1}
    X_1(z) &= \sum_{n=1}^\infty \frac{8\pi n \sin 2\pi n x }
    { e^{2\pi ny} + e^{-2\pi ny} - 2 \cos 2\pi n x} \\ \label{Y1}
    Y_1(z) &= \frac{1}{y} - \frac{\pi}{3} +
    \sum_{n=1}^\infty \frac{-8 \pi n e^{-2\pi n y} + 8\pi n \cos 2\pi n x }
        { e^{2\pi ny} + e^{-2\pi ny} - 2 \cos 2\pi n x} \\  \label{X0}
    X_0(z) &= \sum_{n=1}^\infty \frac{4\pi n \sin  \pi n (x+1) }
    { e^{\pi ny} + e^{-\pi ny} - 2 \cos \pi n (x+1) } \\   \label{Y0}
    Y_0(z) &= \frac{1}{y} - \frac{\pi}{6} +
    \sum_{n=1}^\infty \frac{-4 \pi n e^{-\pi n y} + 4\pi n \cos \pi n (x+1) }
        { e^{\pi ny} + e^{-\pi ny} - 2 \cos \pi n (x+1)}.    
  \end{align}

  We end this section with two formulas that relate 
  $f_b$ on the upper half of the unit circle 
  to $f_{1-b}$ on the upper half of the imaginary axis.

  \begin{lemma}
    \label{l-circle-im}
    Let the upper half of the unit circle be parametrized by
    $u+ i \sqrt{1-u^2}$, $u \in (-1,1)$. Then
    $\frac{\sqrt{1-u^2}}{1-u} i$ parametrizes the upper half of the
    imaginary axies, and 
  \begin{align}
  \label{line-to-circle-X}
  X_b(u + i\sqrt{1-u^2}) &= \frac{\sqrt{1-u^2}}{1-u}  Y_{1-b}
  \Big(\frac{\sqrt{1-u^2}}{1-u} i \Big) \\  \label{line-to-circle-Y}
  Y_b(u + i \sqrt{1-u^2}) &= \frac{-u}{1-u}  Y_{1-b}
  \Big(\frac{\sqrt{1-u^2}}{1-u} i\Big) 
  \end{align}
  hold for $u \in (-1,1)$.
    \end{lemma}

\begin{proof}
Consider the transform in Lemma \ref{l-dual}, $z \rightarrow w = \frac{z-1}{z+1}$.  With $z=x+yi$ and
$w=u+vi$, 
\begin{equation}
  \label{key-transform}
  u = \frac{x^2+y^2-1}{(x+1)^2+y^2}, \ \ v=\frac{2y}{(x+1)^2+y^2}.
\end{equation}
Conversely,
\begin{equation}
  \label{key-transform-1}
  x = \frac{1-u^2-v^2}{(1-u)^2+v^2}, \ \ y=\frac{2v}{(1-u)^2+v^2}.
\end{equation}
Differentiate $f_b(w) = f_{1-b} (z)$ with respect to $u$ and $v$ to find
\begin{align*} 
  &X_b(w) = X_{1-b}(z) \frac{\partial x}{\partial u} + Y_{1-b}(z)
  \frac{\partial y}{\partial u} \\
   & Y_b(w) = X_{1-b}(z) \frac{\partial x}{\partial v} + Y_{1-b}(z)
  \frac{\partial y}{\partial v}. 
\end{align*}
When $w$ is on the unit circle, $z$ is on the imaginary axis. Since $f_{1-b}$
is invariant under the reflection about the imaginary axis,
$X_{1-b}(z) =0$ on the imaginary axis. Also 
\begin{equation}
   \frac{\partial y}{\partial u}\Big |_{|w|=1}  = \frac{\sqrt{1-u^2}}{1-u}, 
  \ \ \frac{\partial y}{\partial v}\Big |_{|w|=1}  = \frac{-u}{1-u}
\end{equation}
from which the lemma follows.
\end{proof}

\section{$f_b$ on imaginary axis}
\setcounter{equation}{0}

The behavior of $f_b$ on the imaginary axis is studied in this section.
Let us record some of the derivatives of $f_1$ and $f_0$ on the
imaginary axis for use here and later.
Let
\begin{equation}
  r = e^{-\pi y}.
  \label{r}
  \end{equation}
Then by \eqref{Y1} and \eqref{Y0}, 
\begin{align}
  Y_1(yi) &= \frac{1}{y} - \frac{\pi}{3} + \sum_{n=1}^\infty
  \frac{8\pi n r^{2n}}{1-r^{2n}} \label{Y1-im}\\
  \frac{\partial Y_1(yi)}{\partial y}
  &= -\frac{1}{y^2} - \sum_{n=1}^\infty
  \frac{16\pi^2 n^2 r^{2n}}{(1-r^{2n})^2} \label{Y1'-im}\\
  \frac{\partial^2 Y_1(yi)}{\partial y^2}
  &= \frac{2}{y^3} + \sum_{n=1}^\infty
  \frac{32\pi^3 n^3 (r^{2n}+r^{4n})}{(1-r^{2n})^3}\label{Y1''-im} \\
  \frac{\partial^3 Y_1(yi)}{\partial y^3}
  &= -\frac{6}{y^4} - \sum_{n=1}^\infty
  \frac{64\pi^4 n^4 (r^{2n}+4r^{4n}+r^{6n})}{(1-r^{2n})^4} \label{Y1'''-im}\\
  Y_0(yi) &= \frac{1}{y} - \frac{\pi}{6} + \sum_{n=1}^\infty
    \frac{4\pi n (-r)^n}{1-(-r)^n} \label{Y0-im} \\
  \frac{\partial Y_0(yi)}{\partial y}
  &= -\frac{1}{y^2} - \sum_{n=1}^\infty
  \frac{4\pi^2 n^2 (-r)^n}{(1 - (-r)^n)^2} \label{Y0'-im} \\
   \frac{\partial^2 Y_0(yi)}{\partial y^2}
  &= \frac{2}{y^3} + \sum_{n=1}^\infty
   \frac{4\pi^3 n^3 \Big ((-r)^n+r^{2n} \Big )}{(1 - (-r)^n)^3}
       \label{Y0''-im} \\
   \frac{\partial^3 Y_0(yi)}{\partial y^3}
  &= -\frac{6}{y^4} - \sum_{n=1}^\infty
   \frac{4\pi^4 n^4\Big((-r)^n+4r^{2n} +(-r)^{3n} \Big)}{(1 - (-r)^n)^4}.
       \label{Y0'''-im}
\end{align}

\begin{lemma}
  \label{l-imaginary}
  For all $b \in \mathbb{R}$,
  \[ f_b(yi) = f_b \Big ( \frac{i}{y}  \Big ), \ \ y>0. \]
  Consequently,
  \[ Y_b(yi) = \Big (- \frac{1}{y^2} \Big )
  Y_b  \Big ( \frac{i}{y} \Big ), \ \ y>0. \]
  In particular \[ Y_b(i)=0. \]
  \end{lemma}
\begin{proof}
  Apply the invariance of $f_b(z)$ under $z \rightarrow -\frac{1}{z}$ with
  $z=yi$, $y>0$. Then differentiate with respect to $y$ and set $y=1$.
  \end{proof}

\begin{lemma}
\label{l-f1-im}
The function $y \rightarrow f_1(yi)$, $y>0$, has one
critical point at $y=1$. Moreover
\begin{align*}
  Y_1(yi)  &>  0 \ \mbox{if} \ y \in (0,1) \\
  Y_1(yi)  &<  0 \ \mbox{if} \ y \in (1,\infty)
\end{align*}
\end{lemma}

\begin{proof}
  Lemma \ref{l-imaginary} asserts that $y=1$ is a critical point of
  $y \rightarrow f_1(y i)$, $y>0$, i.e. $Y_1(i)=0$. Define 
\begin{equation}
 A(z) = \arg (z \eta(z)) = \arg (z) + \arg (\eta(z))
\end{equation}
Note  
\begin{equation}
  \label{log-arg}
  \re(\log(z \eta(z))) = \log |z \eta(z)|, \ \ \im(\log(z
  \eta(z)))=A(z).
  \end{equation}
Hence $A$ is a harmonic function. We consider $A(z)$ in
$U$ and and its closure $\overline{U}_{\mathbb{H}}$ given in
\eqref{U} and \eqref{U-closure} respectively. 

On the imaginary axis, for $y>0$, since $\eta(yi)$ is real and positive,
\begin{equation}
  \label{A-im}
  A(yi) = \arg(yi) + \arg(\eta(yi)) = \frac{\pi}{2} + 0 =
  \frac{\pi}{2},
  \end{equation}
  On the line $x=\frac{1}{2}$, $\arg(\eta(z)) = \frac{\pi}{6}$
  since $e^{2\pi n (\frac{1}{2} + yi)i}$ is real, and 
  \begin{equation}
    A\Big(\frac{1}{2} + yi\Big) = \arctan (2y) + \frac{\pi}{6}.
    \end{equation}
In particular
\begin{equation}
  \label{A-x=1/2}
  A\Big(\frac{1}{2} + yi\Big) > \frac{\pi}{2} \  \ \mbox{if } y >
  \frac{\sqrt{3}}{2}.
  \end{equation}
As $y \rightarrow \infty$ in $z=x+yi$,
\begin{equation}
  \label{A-infty}
  \lim_{y\rightarrow \infty} A(x+yi) = \frac{\pi}{2} + \frac{\pi x}{3} \ \
  \mbox{uniformly with respect to } x \in \Big [0, \frac{1}{2} \Big ].
  \end{equation}
Now consider $A$ on the unit circle.  
By the functional equation \eqref{inv-2}
one has, in polar coordinates $z = r e^{i \theta}$, 
\[ \log (r |\eta(r e^{i\theta}) |) = \log (\frac{1}{r} |\eta
  (-\frac{1}{r} e^{-i\theta}) | ). \]
By the definition of $\eta$, one sees  that $|\eta (- \bar{\zeta})|
= | \eta(\zeta) |$ for all $\zeta \in \mathbb{H}$. Therefore
\[  \log (r |\eta(r e^{i\theta}) |) = \log (\frac{1}{r} |\eta
  (\frac{1}{r} e^{i\theta}) | ). \]
Differentiating the last equation with respect to $r$ and setting
$r=1$ afterwards, one derives 
\[ \frac{\partial }{\partial r} \Big |_{r=1} \log (r |\eta(r
  e^{i\theta}) |) =0. \]
 One of the Cauchy-Riemann equations in polar coordinates for $\log (z
 \eta(z))$ is 
\[ \frac{\partial }{\partial r} \re(\log (z\eta(z))) =
  \frac{1}{r} \frac{\partial}{\partial \theta} \im(\log(z\eta(z))). \]
By \eqref{log-arg}
\[ \frac{\partial }{\partial \theta} A(e^{i\theta}) =0, \]
namely $A$ is constant on the unit circle. Since
$\eta(i)$ is real and positive, $A(i) = \frac{\pi}{2}$. Hence 
\begin{equation}
  \label{A-arc}
  A(z) =\frac{\pi}{2} \ \ \mbox{if} \ |z|=1 \ \mbox{and} \
  \frac{\pi}{3} \leq \arg z \leq \frac{\pi}{2} .
  \end{equation}

By \eqref{A-im}, \eqref{A-x=1/2}, \eqref{A-infty},
\eqref{A-arc}, and the maximum principle, 
\begin{equation}
  A(z) > \frac{\pi}{2}, \ \ z \in U;
  \end{equation}
by the Hopf lemma,
\begin{equation}
  \frac{\partial}{\partial x} \Big |_{x=0, y>1} A(z) >0.
  \end{equation}
By a Cauchy-Riemann equation
\begin{equation}
  Y_1(yi)   = - \frac{\partial}{\partial x} \Big |_{x=0, y>1} A(z) <0,
  \ y \in (1,\infty).
  \label{y>1}
  \end{equation}

For $y \in (0,1)$, by Lemma \ref{l-imaginary},
\begin{equation}
  Y_1(yi) = \Big (-\frac{1}{y^2} \Big )
  Y_1\Big (\frac{i}{y} \Big)>0, \ \ y \in (0,1).
  \label{y<1}
  \end{equation}
This completes the proof.
  \end{proof}

\begin{lemma}
\label{l-f0-im}
The function $ y \rightarrow f_0(yi)$, $y>0$,
has three critical points at $\frac{\sqrt{3}}{3}$, $1$, and $\sqrt{3}$.
Moreover
\begin{align*}
  Y_0(yi)  &>  0 \ \mbox{if} \ y \in (0, \sqrt{3}/3) \\
  Y_0(yi)  &<  0 \ \mbox{if} \ y \in (\sqrt{3}/3, 1) \\
  Y_0(yi)  &>  0 \ \mbox{if} \ y \in (1,\sqrt{3}) \\
  Y_0(yi)  &<  0 \ \mbox{if}  \ y \in (\sqrt{3}, \infty).
\end{align*}
\end{lemma}

\begin{proof}
  The transfrom $z \rightarrow \frac{z-1}{2z-1} \in \Gamma$ maps
  $\frac{1}{2} + y i$ to $\frac{1}{2} + \frac{i}{4y}$, so
  \[ \log | y \eta (\frac{1}{2} + yi) |
  = \log | \frac{1}{4y} \eta(\frac{1}{2} + \frac{i}{4y})|
  \]
  Differentiation with respect to $y$ shows that
  \begin{equation} \label{Y1-1/2-reflect}
    Y_1\Big (\frac{1}{2}+ yi \Big )
    = - \frac{1}{4y^2} Y_1\Big (\frac{1}{2}+ \frac{i}{4y}\Big ).
  \end{equation}
  One consequence of \eqref{Y1-1/2-reflect} is that
  \begin{equation}
    \label{c-1}
    Y_0(i)=\frac{1}{2} Y_1\Big (\frac{1}{2}+ \frac{i}{2} \Big ) =0;
    \end{equation}
  namely that $1$ is a critical point of $y \rightarrow f_0(yi)$.

  The combined transform of
  $z \rightarrow w=-\frac{1}{z} \in \Gamma$ and $w \rightarrow - \bar{w}$
  maps the line $\frac{1}{2} + yi$ to $\frac{2}{4y^2+1} + \frac{4y}{4y^2+1} i$,
  the unit circle centered at $1$. The invariance  of $f_1$ under
  this transform yields
  \[ f_1 \Big (\frac{1}{2}+y i \Big)
    = f_1 \Big (\frac{2}{4y^2+1}+ \frac{4yi}{4y^2+1} \Big )\]
  Differentiation with respect to $y$ shows that
  \[ Y_1\Big (\frac{1}{2}+yi \Big) = X_1(\frac{2}{4y^2+1}+ \frac{4yi}{4y^2+1})
  \frac{\partial }{\partial y} \Big ( \frac{2}{4y^2+1} \Big )
  +  Y_1(\frac{2}{4y^2+1}+ \frac{4yi}{4y^2+1})
  \frac{\partial }{\partial y} \Big ( \frac{4y}{4y^2+1} \Big ). \]
  By \eqref{X1}
  \[ X_1 \Big (\frac{1}{2}+v i \Big) =0, \ \ v>0. \]
  Hence, with $y=\frac{\sqrt{3}}{2}$, one deduces
  \begin{equation}
    \label{c-sqrt3}
    Y_0(\sqrt{3}\, i)=\frac{1}{2} Y_1(\frac{1}{2}+ \frac{\sqrt{3}}{2}i) = 0,
    \end{equation}
  i.e. $\sqrt{3}$ is a critical point of $y \rightarrow f_0(yi)$.
  By \eqref{Y1-1/2-reflect}, $\frac{\sqrt{3}}{3}$ is also a critical point
  of $y \rightarrow f_0(y i)$.

  Now show that
  \begin{equation} 
    \label{Y0-im-4}
    Y_0(yi) <0, \ \mbox{if} \ y \in (\sqrt{3}, \infty).
  \end{equation}
  This fact was established by Chen and Oshita in \cite{chenoshita-2}.
  Here we give a
  more direct alternative proof.
  
  Consider the expression for $\frac{\partial Y_0(y i)}{\partial y}$ in 
\eqref{Y0'-im}. Note that the series
\begin{equation}
  \label{series-1}
    \sum_{n=1}^\infty \frac{n^2(-r)^n}{(1-(-r)^n)^2}
  \end{equation}
is alternating. The only nontrivial property to verify is that the
absolute values of the terms decrease, and this follows from the following
estimate.
\begin{align}
  \frac{n^2r^n}{(1-(-r)^n)^2} - \frac{(n+1)^2r^{n+1}}{(1-(-r)^{n+1})^2}
  & = \frac{(n+1)^2 r^{n+1}}{(1-(-r)^n)^2}
  \Big ( \frac{n^2}{(n+1)^2 r} - \frac{(1-(-r)^n)^2}{(1-(-r)^{n+1})^2} \Big )
  \nonumber \\ & \geq \frac{(n+1)^2 r^{n+1}}{(1-(-r)^n)^2}
  \Big ( \frac{e^{\sqrt{3}\pi}}{4} - \frac{(1+e^{-\sqrt{3} \pi})^2}
       {(1-e^{-2\sqrt{3}})^2} \Big ) \nonumber \\
       & = \frac{(n+1)^2 r^{n+1}}{(1-(-r)^n)^2} \  \times 56.68...  >0.
       \label{alt-3}  
  \end{align}
This allows us to estimate  $\frac{\partial Y_0(y i)}{\partial y}$  as follows
\begin{align}
  \frac{\partial Y_0(y i)}{\partial y} & < -\frac{1}{y^2}
  + \frac{ 4\pi^2 e^{-\pi y}} {(1+e^{-\pi y})^2}  \nonumber \\
  & < -\frac{1}{y^2} +  4\pi^2 e^{-\pi y} \nonumber \\
  & = \frac{1}{y^2}  \Big ( -1 + 4\pi y^2 e^{-\pi y}   \Big ) \nonumber \\
  & \leq  \frac{1}{y^2}  \Big ( -1 + 4\pi (\sqrt{3})^2 e^{-\pi \sqrt{3}} \Big )
     \nonumber \\
  &  \leq  \frac{1}{y^2} \ \times ( - 0.8388...) < 0. \label{Y0'-im-3}
  \end{align}
Here to reach the fourth line, one notes that
\[ (-1 + 4\pi y^2 e^{-\pi y})' = 4 \pi e^{-\pi y} y (2- \pi y) <0, \ \
\mbox{if} \ y > \sqrt{3}. \]
Since $Y_0(\sqrt{3} \, i)=0$, \eqref{Y0'-im-3} implies \eqref{Y0-im-4}.

  By \eqref{Y0-im-4} and Lemma \ref{l-imaginary}, one deduces
  \begin{equation} \label{Y0-im-1}
    Y_0(yi) > 0, \ \ \mbox{if} \ \ y \in (0,\frac{\sqrt{3}}{3}).
  \end{equation}

   Next consider $Y_0(yi)$ for $y \in (1, \sqrt{3})$.
 By \eqref{Y0''-im}
  \begin{equation}
    \frac{\partial^2  Y_0(y i)}{\partial y^2} 
  = \frac{2}{y^3} +  \sum_{n=1}^\infty \frac{4 \pi^3 n^3 ((-r)^n +
  r^{2n})}{(1- (-r)^n)^3 }, \ \ r=e^{-\pi y}.
  \label{Y1''-1/2}
  \end{equation}
  It turns out that the series
  \begin{equation}
   \sum_{n=1}^\infty  \frac{n^3 (-r)^n} {(1- (-r)^n)^3 }
    \label{alt-2}
  \end{equation}
  which is part of \eqref{Y1''-1/2}
  is alternating. To see that
  the absolute values of the terms in \eqref{alt-2} decrease, note 
  \[ \qquad \frac{n^3 r^n} {(1- (-r)^n)^3 }
    - \frac{(n+1)^3 r^{n+1}} {(1- (-r)^{n+1})^3 }  
     = \frac{(n+1)^3 r^{n+1}}{(1-(-r)^n)^3}
    \Big [ \frac{n^3}{(n+1)^3 r} - \frac{(1-(-r)^n)^3}{(1-(-r)^{n+1})^3}
      \Big ]  
    \]
  and it suffices to show that the quantity in the brackets is positive.
  For $y > 1$,
  \[    \frac{n^3}{(n+1)^3 r} - \frac{(1-(-r)^n)^3}{(1-(-r)^{n+1})^3}  
      >  \frac{e^{\pi}}{8} -
       \frac{(1 + e^{-\pi})^3 }
      {  (1- e^{-2\pi})^3} 
       =  2.8925... - 1.1417...  \ > \ 0. \]
An upper bound for \eqref{alt-2} is available if one chooses two terms
from the series:
\begin{equation}
 \sum_{n=1}^\infty  \frac{n^3 (-r)^n} {(1- (-r)^n)^3 }
  < \frac{-r}{(1+r)^3} + \frac{8 r^2}{(1-r^2)^3}. \label{alt-2-2}
\end{equation}
Then \eqref{Y1''-1/2} becomes
\begin{align}
  \frac{\partial^2 Y_0(y i) }{\partial y^2}   & <
  \frac{2}{y^3} - \frac{4 \pi^3 r}{(1+r)^3}
  + \frac{32\pi^3 r^2}{(1-r^2)^3}
  + 4\pi^3 \sum_{n=1}^\infty \frac{ n^3 r^{2n}}{(1- (-r)^n)^3 } \nonumber \\
  & < \frac{2}{y^3} - \frac{4 \pi^3 r}{(1+r)^3}
  + \frac{32\pi^3 r^2}{(1-r^2)^3}
  + \frac{4 \pi^3}{(1-r)^3} \sum_{n=1}^\infty n^3 r^{2n} \nonumber \\
  &=  \frac{2}{y^3} - \frac{4 \pi^3 r}{(1+r)^3}
  + \frac{32\pi^3 r^2}{(1-r^2)^3} +
  \frac{4 \pi^3 r^2 (1+4r^2+r^4) }{(1-r)^3(1-r^2)^4} \nonumber \\
  & < \frac{2}{y^3} - \Big [ \frac{4\pi^3}{(1+e^{-\pi})^3} \Big ] r
  + \Big [ \frac{32\pi^3}{(1-e^{-2\pi})^3}
    +\frac{4\pi^3 (1+4e^{-2\pi} + e^{-4\pi})}{(1-e^{-\pi})^3
      (1-e^{-2\pi})^4}\Big ] s^2 \nonumber  \\
  &=  \frac{2}{y^3} - A_1 r + A_2 r^2  \nonumber \\
  &= r \, \kappa(y), \label{Y1''-1/2-2}
\end{align}
where we have used the summation formula
\begin{equation} \label{sum-1} \sum_{n=1}^\infty n^3 t^n =
t \Big ( t \Big [ t \Big ( \frac{1}{1-t} \Big )_t \Big ]_t \Big )_t 
= \frac{t(1+4t+t^2)}{(1-t)^4}, \ \ |t|<1
\end{equation}
to reach the third line, $A_1$ and $A_2$ are given by
\begin{equation}
  A_1 =  \frac{4\pi^3}{(1+e^{-\pi})^3}=109.24...      , \ \ 
  A_2 = \frac{32\pi^3}{(1-e^{-2\pi})^3}
    +\frac{4\pi^3 (1+4e^{-2\pi} + e^{-4\pi})}{(1-e^{-\pi})^3
      (1-e^{-2\pi})^4} = 1,141.50...
  \label{A1A2}
\end{equation}
and $\kappa$ is 
\begin{equation}
  \kappa(y) = \frac{2}{y^3 r} - A_1 + A_2 r
  = \frac{2e^{\pi y} }{y^3} - A_1 + A_2 e^{-\pi y}. 
  \label{kappa}
\end{equation}
Regarding $\kappa$, one finds
\begin{align*}
  \kappa''(y) 
  & = e^{\pi y} \big ( 2 \pi^2 y^{-3} -12 \pi y^{-4} + 24 \pi y^{-5} \big )
  + \pi^2 A_2 e^{-\pi y} \\
  & = 2 e^{\pi y} y^{-5} \big ( (\pi y - 3)^2 +3  \big )+
   \pi^2 A_2 e^{-\pi y} \ > \ 0, 
\end{align*}
and
\[ \kappa(1) = -13.63...<0, \ \
   \kappa(\sqrt{3}) = -15.47...<0. \]
Hence
\begin{equation}
  k(y) <0, \ \ y \in [ 1, \sqrt{3}], 
  \label{kappa-2}
  \end{equation}
and by \eqref{Y1''-1/2-2}
\begin{equation}
  \frac{\partial^2  Y_0(y i) }{\partial y^2}  
  <0, \ \ y \in [1, \sqrt{3}]. 
  \label{Y1''-1/2-3}
  \end{equation}
Since  $Y_0 (1)
=Y_0(\sqrt{3})=0$ by \eqref{c-1} and
\eqref{c-sqrt3}, \eqref{Y1''-1/2-3} implies 
  \begin{equation}
    \label{Y0-im-3}
    Y_0(yi) > 0, \ \ \mbox{if} \ \ y \in
    (1, \sqrt{3}).
  \end{equation}

  By \eqref{Y1-1/2-reflect}, \eqref{Y0-im-3} implies
   \begin{equation}
    \label{Y0-im-2}
    Y_0(yi) < 0, \ \ \mbox{if} \ \ y \in
    \Big (\frac{\sqrt{3}}{3},1 \Big ).
   \end{equation}
   The lemma follows from \eqref{Y0-im-1}, \eqref{Y0-im-2},
   \eqref{Y0-im-3}, and \eqref{Y0-im-4}.
\end{proof}

For $b$ between $(0,1)$, the next lemma shows  that the shape of $f_b$
is similar to $f_0$ if $b$ is small and similar to $f_1$ if $b$ is large.
The borderline is  $B$ given by 
\begin{equation}
  \label{B}
  B = \frac{\frac{\partial Y_0(i)}{\partial y}} 
  {\frac{\partial Y_0(i)}{\partial y} - \frac{\partial Y_1(i)}{\partial y}}
 = \frac{0.2982...}{ 0.2982...- (-1.298...)} = 0.1867...
\end{equation}
 The numerical values in \eqref{B}
 are computed from the series \eqref{Y1'-im} and \eqref{Y0'-im}. 
 One interpretation of $B$ is that if $b=B$, the second derivative of
 $y \rightarrow f_B(y i)$ vanishes at $y=1$, i.e. 
 \begin{equation}
   \label{interp}
     \frac{\partial Y_B(i)}{\partial y}=0.
   \end{equation}

\begin{lemma}
\label{l-fb-im}
The following properties hold for $y \rightarrow f_b(bi)$, $y \in (0,\infty)$. 
\begin{enumerate}
\item When $b \in [0, B)$, the function $y \rightarrow f_b(yi)$, $y>0$,
  has exactly three critical points at $\frac{1}{q_b}$, $1$, and
  $q_b$, where $q_b \in (1, \sqrt{3}]$. Moreover
    \begin{enumerate}
    \item $Y_b(yi) >0$ if $ y \in (0, \frac{1}{q_b})$,
    \item $Y_b(yi) < 0$ if $ y \in (\frac{1}{q_b},1)$,
    \item $Y_b(yi) >0$ if $ y \in (1, q_b)$,
      \item  $Y_b(yi) < 0$ if $ y \in (q_b,\infty)$.
      \end{enumerate}
  As $b$ increases from $0$ to $B$,
   $q_b$ decreases from $\sqrt{3}$ towards $1$.
  \item When $b \in [B, 1]$, the function $y \rightarrow f_b(0,y)$, $y>0$,
 has only one critical point at $1$, and
   \begin{enumerate}
    \item $Y_b(yi)>0$ if $y \in (0,1)$,
      \item $Y_b(yi)<0$ if $y \in (1,\infty)$.
      \end{enumerate}
\end{enumerate}
\end{lemma}

\begin{proof}
  The shapes of $f_1$ and $f_0$ are already established in Lemmas
  \ref{l-f1-im} and \ref{l-f0-im}. 
  These lemmas imply that
  $y=1$ is a critical point of $y \rightarrow f_b(y i)$, $y>0$,  i.e.
  \begin{equation}
    \label{crit-1}
    Y_b(i)=0, \ \ \mbox{for all} \ b \in [0,1],
    \end{equation}
  and moreover, 
  \begin{equation}
    Y_b(yi)<0, \ \ \mbox{if} \ b \in (0,1] \ \mbox{and} \ y \geq \sqrt{3}.
  \end{equation}

  To study $Y_b(yi)$ for $y \in (1, \sqrt{3})$ and $b \in (0,1)$, write $Y_b =
  b Y_1 + (1-b) Y_0$ as
  \begin{equation}
    \label{Yb-alt}
    Y_b(yi) = b Y_1(yi) \Big ( 1+ \big ( \frac{1-b}{b} \big )
     \frac{Y_0(yi)}{Y_1(yi)} \Big ).
  \end{equation}
  Recall $Y_1(y i) <0$ for $y>1$ from Lemma \ref{l-f1-im}. 
  Regarding the quotient $\frac{Y_0(yi)}{Y_1(yi)}$,
  since $Y_0(i)=Y_1(i)=0$, $\frac{Y_0(i)}{Y_1(i)}$ is understood as the limit
  \begin{equation}
    \label{LHospital}
    \frac{Y_0(i)}{Y_1(i)} = \lim_{y \rightarrow 1} \frac{Y_0(yi)}{Y_1(yi)}
    = \frac{\frac{\partial Y_0(i)}{\partial y}}
    {\frac{\partial Y_1(i)}{\partial y}}
     = \frac{0.2982...}{-1.298... }= -0.2297...<0
  \end{equation}
  evaluated by L'Hospital's rule.
    Since $Y_0(\sqrt{3} \,i) =0$ by Lemma \ref{l-f0-im},
  \begin{equation}
    \label{quo-sqrt3}
     \frac{Y_0(\sqrt{3}\, i)}{Y_1(\sqrt{3}\,i)} =0.
  \end{equation}
  Lemmas \ref{l-f1-im} and \ref{l-f0-im} also assert that
  $Y_1(yi) < 0$ and $Y_0(yi) >0$ if $y \in (1,\sqrt{3})$, so 
  \begin{equation}
    \label{quo-between}
     \frac{Y_0(y i)}{Y_1(yi)} < 0, \ \  y \in (1, \sqrt{3}).
  \end{equation}
  However the most important property of this quotient is its monotonicity
  on $(1, \sqrt{3})$, namely
  \begin{equation}
    \label{quo-mono}
    \frac{\partial }{\partial y} \Big ( \frac{Y_0(y i)}{Y_1(yi)} \Big )
    > 0, \ \ y \in (1, \sqrt{3}).  
    \end{equation}
  The proof of \eqref{quo-mono} is long and brute force. We leave it in the
  appendix. The first time reader may wish to skip this part. 

  Return to \eqref{Yb-alt}. Since $Y_1(yi) <0$ on $(1, \infty)$ and
  $\frac{1-b}{b} \in (0, \infty)$ when $b \in (0,1)$,
  \eqref{quo-mono} implies that
  $Y_b(yi)$ can have at most one zero in $(1, \sqrt{3})$
  at which $Y_b(y i)$ changes sign.  Because of \eqref{quo-sqrt3},
  \begin{equation}
    1 +    \Big ( \frac{1-b}{b} \Big )
    \frac{Y_0(\sqrt{3}\, i)}{Y_1(\sqrt{3}\,i)} = 1+0>0, \ \ b \in (0,1). 
  \end{equation}
  Hence $Y_b(y i)$ admits a zero in $(1, \sqrt{3})$  if and only if
  \begin{equation}
    \label{iff}
      1+ \Big ( \frac{1-b}{b} \Big )
     \frac{Y_0(i)}{Y_1(i)} <0. 
    \end{equation}
  The condition \eqref{iff} is equivalent to
  \begin{equation}
    \label{iff-2}
      b < B
    \end{equation}
  by \eqref{LHospital}.  We denote this zero in $(1,\sqrt{3})$
  of $Y_b(yi)$ by $q_b$ when $ b \in (0, B)$.
  It is also clear from \eqref{Yb-alt} and \eqref{quo-mono} that
  as $b$ increases from $0$ to $B$, $q_b$ decreases monotonically
  from $\sqrt{3}$ towards $1$. This proves parts 1(c), 1(d), and
  2(b) of the lemma. The remaining parts follow from
  Lemma \ref{l-imaginary}.
  \end{proof}

\section{$f_b$ on upper half plane}
\setcounter{equation}{0}

We start with a study of the singular point $z=1$. Recall the set
$W$ from \eqref{W}.

\begin{lemma}
  \label{l-singular}
   \[ \limsup_{W \ni z \rightarrow 1} X_b(z) =0. \]
\end{lemma}

\begin{proof}
Note that  if $z=x+yi \in W$, then, when $y < 1$,
  \begin{equation}
    \label{x-range}
    0< 1-x < 1 - \sqrt{1-y^2}.
  \end{equation}
  So $W \ni z \rightarrow 1$ is equivalent to that
  $z \in W$ and $y \rightarrow 0$.

We first show that
  \begin{equation}
   \label{X1-singular}
   \limsup_{W \ni z \rightarrow 1} X_1(z) \leq 0.
  \end{equation}
  Namely, 
  for every $\epsilon >0$ there exists $\delta >0$
  such that if $z=x+yi \in W$ and $y < \delta$, then $X_b(z) < \epsilon$.

Recall 
  \[
    X_1(z) = \sum_{n=1}^\infty \frac{8\pi n \sin 2\pi n x }
    { e^{2\pi ny} + e^{-2\pi ny} - 2 \cos 2\pi n x}
  \]
from \eqref{X1}. 
Separate this infinite sum  into two parts according to
  whether $n y^2 < \frac{1}{2}$ or  $ny^2 \geq \frac{1}{2}$. Write
  \begin{align}
  a_n(z) &= \frac{8\pi n \sin 2\pi n x }
    { e^{2\pi ny} + e^{-2\pi ny} - 2 \cos 2\pi n x} \\
 A(z) &= \sum_{n< 1/(2y^2)} a_n(z) \\
 \widetilde{A}(z) &= \sum_{n\geq 1/(2y^2)} a_n(z) 
\end{align}
so that
\begin{equation}
\label{AtA}
  X_1(z) = A(z) + \widetilde{A}(z). \\
\end{equation}

Consider the case $n < \frac{1}{2y^2}$.
  Since $1-\sqrt{1-y^2} < y^2$ when $y \in (0,1)$ , \eqref{x-range} implies  
  \begin{equation}
    \label{x-range-2}
    0< 1-x < y^2.
    \end{equation}
Then $0< 2 \pi n (1-x) < 2\pi n y^2 < \pi$  and
$\sin 2\pi n x = - \sin 2\pi n (1-x) < 0$. Hence, every term
$a_n(z)$ in $A(z)$ is negative, and consequently
\begin{equation}
\label{A}
 A(z) < 0. 
\end{equation}

When $n \geq \frac{1}{2y^2}$,
\begin{equation}
\label{estimate}
 |a_n(z) | \leq \frac{8\pi n}{e^{2\pi n y} -2}
 \leq \frac{1}{e^{\pi ny}}.
\end{equation}
To see the last inequality, note
\[  \frac{8\pi n}{e^{2\pi n y} -2} \leq 
   \frac{8\pi n}{e^{2\pi n y} -2} \big ( 2n y^2 \big )
    = \frac{16 \pi (ny)^2}{e^{2\pi ny}-1}.
\]
There exists $t_0>0$ such that for all $t > t_0$,
$\frac{16 \pi t^2}{e^{2\pi t}-1} \leq \frac{1}{e^{\pi t}}$. Since
$n \geq \frac{1}{2y^2}$, $n y \geq \frac{1}{2y}$. By choosing
$y < \frac{1}{2 t_0}$, we have $n y > t_0$ and the last inequality of
\eqref{estimate} follows.
Then
\begin{equation}
\label{estimate-2}
  | \widetilde{A} (z) |
  \leq \sum_{n\geq 1/(2y^2)} \frac{1}{e^{\pi n y}}
  \leq \frac{e^{-\pi y \left [\frac{1}{2y^2} \right]}}{1- e^{-\pi y}}
  \rightarrow 0 \ \ \mbox{as} \ y \rightarrow 0
\end{equation}
where $\left[\frac{1}{2y^2} \right]$ is the integer part of $\frac{1}{2y^2}$.
The claim \eqref{X1-singular} now follows from \eqref{A} and
\eqref{estimate-2}. 

Next we claim that
\begin{equation}
\label{X0-singular}
\limsup_{W \ni z \rightarrow 1}X_0(z) \leq 0.
\end{equation}
This is proved by a similar argument whose details are omitted. 
By \eqref{X1-singular} and \eqref{X0-singular} we obtain that
\begin{equation}
\label{Xb-singular}
\limsup_{W \ni z \rightarrow 1}X_b(z) \leq 0.
\end{equation}

From the series \eqref{X1} and \eqref{X0},
\begin{equation}
\label{Xb-x=1}
X_b(1+yi) =0. \ \ y>0
\end{equation}
This turns \eqref{Xb-singular} to
\begin{equation}
\label{Xb-singular-2}
\limsup_{W \ni z \rightarrow 1}X_b(z) = 0,
\end{equation}
proving the lemma. 
\end{proof} 

Recall that for $b \in [0, B)$, the largest of the three critical points 
  of the function $y \rightarrow f_b(yi)$, $y > 0$, is denoted $q_b$.
  This $q_b$ is a maximum and $1 < q_b \leq \sqrt{3}$. 
  By convention
  if $b \in [B, 1]$, we set $q_b=1$ which is the unique critical point
  (a maximum) of $y \rightarrow f_b(yi)$, $y > 0$.

If $b \in [B, 1]$, then $1-b \in [0, 1-B]$ and $q_{1-b}$ is defined as above.
The transform $z=x+y i \rightarrow w=u+vi$ in \eqref{key-transform} sends
  the point $z=q_{1-b}i$ to $w = \frac{q_{1-b}^2-1 + 2q_{1-b} i}{q_{1-b}^2+1}$.
  Define
\begin{equation}
  \label{pq}
   p_b =  \frac{q_{1-b}^2-1}{q_{1-b}^2+1}.
  \end{equation}
Then
\begin{equation}
  \label{pq-2}
  \frac{q_{1-b}^2-1 + 2q_{1-b} i}{q_{1-b}^2+1} = p_b + i \sqrt{1-p_b^2}.
\end{equation}

\begin{lemma}
  Let $b \in [0,1-B]$ and $W$ be given in \eqref{W}.
  Then $X_b(z) <0$ for all $z \in W$.
  \label{l-monotone}
\end{lemma}

\begin{proof}

  From \eqref{X1} and \eqref{X0} one deduces 
  \begin{equation}
    \label{X-line-x=0,1}
    X_b(yi) =0 \ \ \mbox{and} \ \ X_b(1+yi)=0, \ \ y>0.
    \end{equation}
  Also
    \begin{equation}
      \label{X-y=inf}
      \lim_{y \rightarrow \infty} X_b(z) = 0
      \ \ \mbox{uniformly with respect to} \ x \in\mathbb{R}. 
      \end{equation}
  
  On the unit circle, we know from
  \eqref{line-to-circle-X}
  \begin{equation}
    \label{line-to-circle-X-2}
    X_b(x + i\sqrt{1-x^2}) = \frac{\sqrt{1-x^2}}{1-x}
    Y_{1-b}\Big(\frac{\sqrt{1-x^2}}{1-x} \, i \Big), \ \ x \in (-1,1).
  \end{equation}
  When $b \in [0, 1-B]$, $ 1-b \in [B, 1]$. By Lemma \ref{l-fb-im}.2,
  \[  Y_{1-b}\Big(\frac{\sqrt{1-x^2}}{1-x} \, i \Big) <0, \ \ \mbox{if} \ \
  x \in (0,1). \]
    This shows 
    \begin{equation}
      \label{X-circle}
            X_b(x + i\sqrt{1-x^2}) <0, \ \ \mbox{if} \ \  x \in (0,1).
      \end{equation}
 The lemma follows from
\eqref{X-line-x=0,1}, \eqref{X-y=inf}, \eqref{X-circle},
and Lemma \ref{l-singular} by the maximum principle.
\end{proof}

We are now ready to prove the main theorem. 

\begin{proof}[Proof of Theorem \ref{t-main}]

  \

  \
  
\noindent Claim 1. 
  Let $b \in [0, 1-B]$. Then $f_b(z)$ on the upper half plane
  is maximized at $q_b i$ and the points in the orbit of $q_b i$
  under the group ${\cal G}$.

\

The second plot of Figure \ref{f-W} demonstrates our argument. 
By Lemma \ref{l-inv-group}.3 it suffices to consider $f_b$
in $\overline{W}_{\mathbb{H}}$. In $W$ 
Lemma \ref{l-monotone} asserts that $f_b$ is strictly decreasing in
the horizontal direction, so it can only attain a maximum
in $\overline{W}_{\mathbb{H}}$
on the part of the unit circle in the first quadrant, i.e.
$\{ w \in \mathbb{C}: \ |w|=1, \ 0<\re w<1, \ \im w >0\}$, or on
the part of the imaginary axis above $i$, i.e.
$\{ z \in \mathbb{C}: \ \re z =0, \ \im z \geq 1 \}$.

  First rule out the unit circle. 
  By Lemma \ref{l-dual}
  \begin{equation}
  f_b(w)=f_{1-b}(z), \ z \in \mathbb{H}, \ w=\frac{z-1}{z+1} \in \mathbb{H}. 
    \label{dual}
    \end{equation}
  Take $z=yi$ to be on the imaginary axis. Then
  \[ w= \frac{y^2-1}{y^2+1} + \frac{2y}{y^2+1} i \]
  is on the unit circle. As $z$ moves from $i$ to $\infty$ upward
  along the imaginary axis, $w$ moves from
  $i$ to $1$ clockwise along the unit circle. 
  When $b \in [0,1-B]$, $1-b \in [B,1]$.
  Since $y \rightarrow f_{1-b}(yi)$
  is strictly decreasing for $y \in (1, \infty)$
  by Lemma \ref{l-fb-im}.2, $f_b(w)$ is strictly decreasing when
  $w$ moves from $i$ to $1$ clockwise along the unit circle. Then
  $f_b$ cannot attain a maximum on
  $\{ z\in \mathbb{C}: \ |w|=1, \ 0<\re w<1, \ \im w>0\}$.

  Therefore in $\overline{W}_{\mathbb{H}}$, $f_b$ can only achieve a maximum
  on $\{ z \in \mathbb{C}: \ \re z =0, \ \im z \geq 1 \}$.
  By Lemma \ref{l-fb-im}.1, it does so at $q_b i$. This proves Claim 1.

  \ 

  By Lemma
  \ref{l-fb-im}.1 and the convention that $q_b=1$ if $b \in [B,1]$,
  three possibilities exist for $q_b$ when $b \in [0, 1-B]$.
  When $b=0$, $q_b=\sqrt{3}$, which proves part 1 of the theorem.
  When $b \in (0,B)$, $q_b \in (1, \sqrt{3})$, which proves part 2 of
  the theorem.
  When $b\in [B,1-B]$, $q_b =1$, which proves part 3 of the theorem. 

  \

  Now consider the case $B \in (1-B,1]$. 

\

\noindent Claim 2. 
  If $b \in (1-B, 1]$, then $f_b$ on the upper half plane
  is maximized at $p_b +i \sqrt{1-p_b^2}$ and the points in its orbit
  under the group ${\cal G}$.

\

 By Lemma \ref{l-dual}, the duality property, we have
\[ f_b(w) = f_{1-b} (z), \ \  z \in \mathbb{H} \ \ \mbox{and} \ \ 
w=\frac{z-1}{z+1} \in \mathbb{H} \]
If $w_\ast$ maximizes $f_b$, then $z_\ast = \frac{w_\ast+1}{-w_\ast+1}$
maximizes $f_{1-b}$. Since $b \in (1-B, 1]$, $1-b \in [0, B)$. By
Claim 1, $z_\ast = q_{1-b}i$ or a point in the orbit
of $q_{1-b}i$ under ${\cal G}$.
Under the transform
$w=\frac{z-1}{z+1}$, $z_\ast=q_{1-b}i$ corresponds to
\begin{equation} w_\ast=\frac{q_{1-b}i-1}{q_{1-b}i+1}
  =  \frac{q_{1-b}^2-1 + 2q_{1-b} i}{q_{1-b}^2+1}
  = p_b + i \sqrt{1-p_b^2}
  \label{wast}
\end{equation}
by \eqref{pq-2}. This proves Claim 2.

\

When $b\in (1-B,1)$, $q_{1-b} \in (1, \sqrt{3})$ by
Lemma \ref{l-fb-im}.1. Then by \eqref{pq-2}, $p_b+i \sqrt{1-p_b^2}$
identified as a maximum of $f_b$ in Claim 2 is in 
$\{ z \in \mathbb{C}: \ |z|=1,
\ \frac{\pi}{3} < \arg z < \frac{\pi}{2} \}$. This proves
part 4 of the theorem. Finally when $b=1$, $q_0=\sqrt{3}$ and
\begin{equation}
  \label{b=1}
  p_1 + i \sqrt{1-p_1^2}=\frac{1+\sqrt{3} i}{2}
\end{equation}
by \eqref{pq-2}. This proves part 5 of the theorem. 
\end{proof}

\appendix
\renewcommand{\theequation}{A.\arabic{equation}}
\setcounter{equation}{0}

\vskip 1cm

\noindent{\Large \bf Appendix}

\

\begin{proof}[Proof of \eqref{quo-mono}]

Here we prove the monotonicity of the function
$y \rightarrow \frac{Y_0(yi)}{Y_1(y i)}$, $y \in (1, \sqrt{3})$, by
showing
\begin{equation}
  \label{mono}
 \frac{\partial }{\partial y} \Big ( \frac{Y_0(yi)}{Y_1(y i)} \Big )
  >0, \ y \in (1, \sqrt{3}).  
\end{equation}
Our proof uses the following L'Hospital like criterion for
monotonicity. See \cite{anderson-vamanamurthy-vuorinen} for more
information on this trick.

\

\noindent Claim.
\begin{equation}
  \frac{\partial }{\partial y} \Big ( \frac{Y_0(yi)}{Y_1(y i)} \Big ) >0 
  \ \mbox{on}  
  \  (1, \sqrt{3}) \ \mbox{if} \ y \rightarrow
  \frac{\frac{\partial Y_0(yi)}{\partial y}}{
    \frac{\partial Y_1(y i)}{\partial y}}
  \ \mbox{is strictly increasing on} \ (1,\sqrt{3}).
  \label{gromov}
  \end{equation}

\

Let $y \in (1, \sqrt{3})$. There exist $y_1 \in (1, y)$ and $y_2\in (1,y)$
such that
\begin{align*}
  \frac{\partial }{\partial y} \Big ( \frac{Y_0(yi)}{Y_1(y i)} \Big )
  & = \frac{\frac{\partial Y_0(yi)}{\partial y}
    Y_1(yi) - Y_0(y i) \frac{\partial Y_1(yi)}{\partial y} }{Y_1^2(yi)} \\
  & = \frac{\frac{\partial Y_1(yi)}{\partial y} }{Y_1(y i)}
  \Big (  \frac{ \frac{\partial Y_0(yi)}{\partial y} }
       { \frac{\partial Y_1(yi)}{\partial y} }
       - \frac{Y_0(yi)}{Y_1(yi)}\Big ) \\
       & =  \frac{\frac{\partial Y_1(yi)}{\partial y} }{Y_1(y i)-Y_1(i)}
       \Big (  \frac{ \frac{\partial Y_0(yi)}{\partial y} }
            { \frac{\partial Y_1(yi)}{\partial y} }
            - \frac{Y_0(yi)-Y_0(i)}{Y_1(yi)-Y_1(i)}\Big ) \\
            & = \frac{\frac{\partial Y_1(yi)}{\partial y} }
            {\frac{\partial Y_1(y_1i)}{\partial y} (y-1) }
       \Big (  \frac{ \frac{\partial Y_0(yi)}{\partial y} }
            { \frac{\partial Y_1(yi)}{\partial y} }
            - \frac{ \frac{\partial Y_0(y_2i)}{\partial y} }
            { \frac{\partial Y_1(y_2i)}{\partial y} }\Big )
\end{align*}
since $Y_1(i)=Y_0(i)=0$. 
Because $\frac{\partial Y_1(yi)}{\partial y}$ does not change sign in
$(1, \sqrt{3})$,
\begin{equation}
  \label{term1} \frac{\frac{\partial Y_1(yi)}{\partial y} }
        {\frac{\partial Y_1(y_1i)}{\partial y} (y-1) } >0.
        \end{equation}
Moreover, since
\[  y \rightarrow
  \frac{\frac{\partial Y_0(yi)}{\partial y}}{
    \frac{\partial Y_1(y i)}{\partial y}} \]
  is strictly increasing,
  \begin{equation}
    \label{term2}
     \frac{ \frac{\partial Y_0(yi)}{\partial y} }
            { \frac{\partial Y_1(yi)}{\partial y} }
            - \frac{ \frac{\partial Y_0(y_2i)}{\partial y} }
            { \frac{\partial Y_1(y_2i)}{\partial y} } >0.
    \end{equation}
The claim then follows from \eqref{term1} and \eqref{term2}.

\

We proceed to show that
\begin{equation}
  \frac{\partial }{\partial y} \Big (
  \frac{\frac{\partial Y_0(yi)}{\partial y}}{
    \frac{\partial Y_1(y i)}{\partial y}} \Big )
  = \frac{\frac{\partial^2 Y_0(yi)}{\partial y^2}
    \frac{\partial Y_1(yi)}{\partial y} -
    \frac{\partial Y_0(yi)}{\partial y}
    \frac{\partial^2 Y_1(yi)}{\partial y^2}
  }
  { \Big ( \frac{\partial Y_1(y i)}{\partial y} \Big )^2}
  >0, \ y \in (1, \sqrt{3}).  \label{gromov2}
\end{equation}
Define
\begin{equation} \label{T}
   T(y)= \frac{\partial^2 Y_0(yi)}{\partial y^2}
    \frac{\partial Y_1(yi)}{\partial y} -
    \frac{\partial Y_0(yi)}{\partial y}
    \frac{\partial^2 Y_1(yi)}{\partial y^2}.
  \end{equation}
By \eqref{Y1'-im}, $\frac{\partial Y_1(yi)}{\partial y} < 0$ on $(0, \infty)$.
Therefore to prove \eqref{gromov2}, it suffices to show
\begin{equation}
  T(y) > 0, \ y \in (1, \sqrt{3}). 
  \label{T>0}
 \end{equation}

We divide $(1, \sqrt{3})$ into two intervals: $(1, \beta)$ and
$[\beta, \sqrt{3})$  where $\beta \in (1,\sqrt{3})$ is to be determined.
  First consider $T(y)$ on $(1,\beta)$.
  Lemma \ref{l-imaginary} asserts that for $j=0,1$,
  \[ Y_j(yi) = \Big (- \frac{1}{y^2} \Big ) Y_j\Big (\frac{i}{y} \Big ). \]
  Differentiation shows that
  \begin{align}
    \frac{\partial Y_j(yi)}{\partial y} &= 2 y^{-3}
    Y_j\Big (\frac{i}{y} \Big )+ y^{-4}
    \frac{\partial Y_j \big (\frac{i}{y} \big )}{\partial y} \\
    \frac{\partial^2 Y_j(yi)}{\partial y^2} &=-6 y^{-4}
    Y_j\Big (\frac{i}{y} \Big ) - 6 y^{-5} 
    \frac{\partial Y_j \big (\frac{i}{y} \big )}{\partial y}
    - y^{-6} \frac{\partial^2 Y_j \big (\frac{i}{y} \big )}{\partial y^2}.
    \label{Y''-ref}
    \end{align}
  Taking $y=1$ in \eqref{Y''-ref} and using $Y_j(i)=0$, one obtains
  \begin{equation}
    \label{recursive}
    \frac{\partial^2 Y_j(i)}{\partial y^2}
    = -3 \frac{\partial Y_j(i)}{\partial y}, \ j=1,0.
  \end{equation}
  In particular \eqref{recursive} implies that
  \begin{equation}
    \label{Tat1}
    T(1)=0.
    \end{equation}

  Next consider the derivative of $T$,
  \begin{equation} \label{T'}
     T'(y) = \frac{\partial^3 Y_0(yi)}{\partial y^3}
    \frac{\partial Y_1(yi)}{\partial y} -
    \frac{\partial Y_0(yi)}{\partial y}
    \frac{\partial^3 Y_1(yi)}{\partial y^3}.
    \end{equation}
  It is clear from \eqref{Y1'-im} and \eqref{Y1'''-im} that
  \begin{equation}
    \frac{\partial Y_1(yi)}{\partial y} <0, \ \
    \frac{\partial^3 Y_1(yi)}{\partial y^3} <0, \ \ y >0. \label{T12}
    \end{equation}

  Similar to the argument following \eqref{alt-2}, one finds
  the series in \eqref{Y0'-im} to be alternating when $y>1$. Therefore,
  \begin{equation}
    \frac{\partial Y_0(yi)}{\partial y}
    > -\frac{1}{y^2} + \frac{4\pi^2 e^{-\pi y}}{(1+e^{-\pi y})^2}
    - \frac{16\pi^2 e^{-2 \pi y}}{(1-e^{-2\pi y})^2}
    \label{Y0'-im-2}
    \end{equation}
  We will later choose $\beta\in (1,\sqrt{3})$
  so that when $y=\beta$, the right side
  of \eqref{Y0'-im-2} is positive; namely choose $\beta$ to make
  \begin{equation}
    -\frac{1}{\beta^2} + \frac{4\pi^2 e^{-\pi \beta}}{(1+e^{-\pi \beta})^2}
    - \frac{16\pi^2 e^{-2 \pi \beta}}{(1-e^{-2\pi \beta})^2} > 0.
    \label{eta-1}
  \end{equation}
  Since
  \begin{equation}
    \label{Y0''-1-sqrt3}
    \frac{\partial^2 Y_0(yi)}{\partial y^2} <0, \ \ 
    y \in (1, \sqrt{3})
  \end{equation}
  by \eqref{Y1''-1/2-3}, the condition
  \eqref{eta-1} implies that
  \begin{equation}
     \frac{\partial Y_0(yi)}{\partial y} >0, \ \ y \in (1,\beta).
    \label{T3}
    \end{equation}

  Regarding $\frac{\partial^3 Y_0(yi)}{\partial y^3}$, write it as
  \begin{equation}
     \frac{\partial^3 Y_0(yi)}{\partial y^3} = -\frac{6}{y^4}
     -4\pi^4 \sum_{n=1}^\infty \frac{n^4 (-r)^n}{(1-(-r)^n)^4}
       - 16\pi^4 \sum_{n=1}^\infty \frac{n^4 r^{2n}}{(1-(-r)^n)^4}
       - 4\pi^4 \sum_{n=1}^\infty \frac{n^4 (-r)^{3n}}{(1-(-r)^n)^4}
      \label{Y0'''-series}
  \end{equation}
  where
  \begin{equation}
    \label{r-2}  r=e^{-\pi y}.
    \end{equation}
  as before. Clearly, the second series in \eqref{Y0'''-series}
  is positive for all
  $y>0$. One can also show as before that when $y>1$, the first and the third
  series in \eqref{Y0'''-series} are both alternating. Pick three leading
  terms from the first series and one term from each of the second and the
  third series  to form an upper bound:
  \begin{align}
    \frac{1}{4\pi^4}
    \frac{\partial^3 Y_0(yi)}{\partial y^3} & <  -\frac{6}{4\pi^4 y^4}
    + \frac{r}{(1+r)^4} - \frac{16 r^2}{(1-r^2)^4}
    + \frac{81r^3}{(1+r^3)^4} 
     - \frac{4 r^2}{(1+r)^4} 
     + \frac{r^3}{(1+r)^4} \nonumber \\
     & < -\frac{6}{4\pi^4 y^4}
     + \frac{r}{(1+r)^4} - \frac{16 r^2}{(1-r^2)^4} + 82 r^3
     - \frac{4r^2}{(1+r)^4} \nonumber \\
     & = -\frac{6}{4\pi^4 y^4} +
     \frac{r}{(1-r^2)^4} \Big ( 1- 24 r + 104 r^2 -28 r^3
     - 311 r^4 -4 r^5 + 497 r^6 -328 r^8 + 82r^{10}  \Big ) \nonumber \\
     & < -\frac{6}{4\pi^4 y^4} +
     \frac{r}{(1-r^2)^4} \big ( 1- 24 r + 104 r^2 \big ) \nonumber \\
     & \leq  -\frac{6}{4\pi^4 y^4} +
      \frac{1}{(1-e^{-2\pi})^4} \  r \big ( 1- 24 r + 104 r^2 \big ).
  \end{align}
  Denote the last line by
  \begin{equation}
    \label{sigma}
    \sigma (y) =  -\frac{6}{4\pi^4 y^4} +
    \frac{1}{(1-e^{-2\pi})^4} \Big ( e^{-\pi y}- 24 e^{-2\pi y} +
    104 e^{-3\pi y} \Big )
  \end{equation}
  Compute 
  \begin{align}
    \label{sigma'}
    \sigma'(y) &= \frac{24}{4 \pi^4 y^5}
    +  \frac{\pi e^{-\pi y}}{(1-e^{-2\pi})^4} \Big (-1 + 48 e^{-\pi y}
    - 312  e^{-2\pi y} \Big )
    \end{align}
  and consider the quantity in the parentheses,
  \begin{equation}
     \phi(y) = -1 + 48 e^{-\pi y} - 312  e^{-2\pi y}.
  \end{equation}
  Since $\phi'(y) =\pi e^{-\pi y} (- 48 + 624 e^{-\pi y}) < 0$ if
  $y>1$, 
  \begin{equation}
    \phi(y) > \phi(\beta), \ \ \mbox{if} \ y \in (1,\beta). 
  \end{equation}
  If one can make $\phi(\beta)>0$, namely if
  \begin{equation}
    -1 + 48 e^{-\pi \beta} - 312  e^{-2\pi \beta} >0,
    \label{eta-2}
    \end{equation}
  then
  \begin{equation}
    \phi(y) >0, \ \ y\in (1,\beta).
    \end{equation}
  Consequently, by \eqref{sigma'}
  \begin{equation}
    \label{sigma'-2}
     \sigma'(y) >0, \ \ y \in (1,\beta)
  \end{equation}\
  and
  \begin{equation}
   \sigma(y) < \sigma(\beta), \ \ y \in (1,\beta). 
  \end{equation}
  This shows that
  \begin{equation}
    \frac{\partial^3 Y_0(y i)}{\partial y^3} \leq 4\pi^4 \sigma (\beta),
     \ \ y \in (1,\beta).
    \end{equation}
  If one can choose $\beta$ so that $\sigma(\beta)<0$, namely
  \begin{equation}
    \label{eta-3}
      -\frac{6}{4\pi^4 \beta^4} +
    \frac{1}{(1-e^{-2\pi})^4} \Big ( e^{-\pi \beta}- 24 e^{-2\pi \beta} +
    104 e^{-3\pi \beta} \Big ) <0, 
  \end{equation}
  then
  \begin{equation}
    \label{T4}
    \frac{\partial^3 Y_0(yi)}{\partial y^3} < 0, \ \ y\in (1,\beta).
    \end{equation}

  Following \eqref{T12}, \eqref{T3}, and \eqref{T4}, one has that
  \begin{equation}
    T'(y) > 0, \ \ y \in (1,\beta).
    \label{T'-2}
  \end{equation}
  By \eqref{Tat1}, \eqref{T'-2} implies that
  \begin{equation}
    \label{small}
   T(y) >0, \ \ y \in (1, \beta) 
  \end{equation}
  provided \eqref{eta-1}, \eqref{eta-2}, and \eqref{eta-3} hold.

Next consider $T(y)$ for $y\in [\beta, \infty)$. Introduce 
\begin{align}
  d(y) &= Y_0(yi) - Y_1(yi) =
  \frac{\pi}{6} - \sum_{k=1}^\infty \frac{4\pi(2k-1)r^{2k-1}}{1+r^{2k-1}}
  \label{d} \\
  d'(y) &= \sum_{k=1}^\infty \frac{4\pi^2(2k-1)^2r^{2k-1}}{(1+r^{2k-1})^2}
  \label{d'} \\
  d''(y) &= \sum_{k=1}^\infty \frac{4\pi^3(2k-1)^3(-r^{2k-1}+r^{2(2k-1)})}
    {(1+r^{2k-1})^3}. \label{d''}
  \end{align}
Then by \eqref{Y1'-im}, \eqref{Y1''-im}, \eqref{d'}, and \eqref{d''}, 
\begin{align}
  T(y) &= d''(y) \frac{\partial Y_1(yi)}{\partial y}
  - d'(y) \frac{\partial^2 Y_1(yi)}{\partial y^2} \nonumber \\
  & =  \sum_{k=1}^\infty
  \frac{4\pi^2 (2k-1)^2 r^{2k-1}}{y^2(1+r^{2k-1})^2}
  \Big ( \frac{\pi (2k-1) (1-r^{2k-1})}{1+r^{2k-1}} - \frac{2}{y} 
  \Big ) \nonumber \\
  & \qquad + \sum_{k=1}^\infty \sum_{n=1}^\infty
  \frac{16 \pi^5 (2k-1)^2 (2n)^2 r^{2n+2k-1}}{(1+r^{2k-1})^2 (1-r^{2n})^2}
  \Big (  \frac{(2k-1)(1-r^{2k-1})}{1+r^{2k-1}}
   - \frac{2n (1+r^{2n})}{1-r^{2n}}
   \Big ) \label{T-2} \\
   &= \sum_{k=1}^\infty c_k + \sum_{k=1}^\infty \sum_{n=1}^\infty d_{kn}
   \label{T-c-d}
  \end{align} 
where $c_k$ and $d_{kn}$ are defined by the terms in \eqref{T-2}.

Regarding $c_k$, because, with $y>1$,
\[  \frac{\pi (2k-1) (1-r^{2k-1})}{1+r^{2k-1}} - \frac{2}{y}
\geq \frac{\pi(1-r)}{1+r} -2
> \frac{\pi(1-e^{-\pi})}{1+e^{-\pi}} -2 >0, \]
one has
\begin{equation}
  \label{c-2}
  c_k >0 \ \ \mbox{for all} \ k.
  \end{equation}
The terms $d_{kn}$ has the following property
\begin{equation}
  d_{kn} > 0 \ \ \mbox{if} \ k >n, \ \ \
  d_{kn} <0 \ \ \mbox{if} \ k \leq n.
  \label{d-2}
\end{equation}
To see \eqref{d-2}, define
\begin{equation}
    \rho_j(r) = \frac{j (1+(-r)^j)}{1-(-r)^j}
\end{equation}
so that the quantity in the second parenthesis pair of \eqref{T-2} is
$\rho_{2k-1}(r) - \rho_{2n} (r)$. The claim \eqref{d-2} follows if
$\rho_j(r)$ is increasing with respect to $j$. To this end consider
\begin{equation}
  \label{rho}
  \rho_{j+1}(r) - \rho_j(r)
  = \frac{1-(2j+1) (r+1) (-r)^j + r^{2j+1}}{(1-(-r)^{j+1}) (1-(-r)^j)}.
\end{equation}
This is clearly positive when $j$ is odd since $0<r<1$. When $j$ is even,
denote the numerator in \eqref{rho} by 
\[  \mu_j(r) = 1- (2j+1) (r+1) r^j + r^{2j+1}. \]
Then
\[ \mu_j'(r) = (2j+1) r^{j-1} \big (- j -(j+1) r + r^{j+1} \big ). \]
As $0<r<1$, 
\[ - j -(j+1) r + r^{j+1} < - j -(j+1) r + r = -j - j r <0. \]
Hence $\mu_j'(r) <0$ for $r \in (0,1)$. Moreover
\[ \mu_j(r) > 1-2(2j+1)r^j, \]
so
\[ \mu_j(e^{-\pi}) > 1- 2(2j+1) e^{-\pi j} > 0\]
for all $j \geq 1$. Therefore $\mu_j(r) >0$ since $r \in (0, e^{-\pi})$
and \eqref{d-2} is proved. 

By \eqref{d-2} drop the positive $d_{kn}$'s to bound the double sum
in \eqref{d} from below by
\begin{equation}
  \label{d-3}
  \sum_{l=1}^\infty \sum_{n=1}^\infty d_{kn}
   > \sum_{k=1}^\infty d_{kk} + \sum_{k=1}^\infty \sum_{n=k+1}^\infty d_{kn}.
  \end{equation}
First consider $\sum_{k=1}^\infty d_{kk}$:
\begin{equation}
  \label{dkk}
   d_{kk} = \frac{16 \pi^5 (2k-1)^2 (2k)^2 r^{4k-1}}{(1+r^{2k-1})^2 (1-r^{2k})^2}
  \Big (  \frac{(2k-1)(1-r^{2k-1})}{1+r^{2k-1}}
   - \frac{2k (1+r^{2k})}{1-r^{2k}}
   \Big ).
  \end{equation}
For $y>1$, 
\begin{align*}
  (1+r^{2k-1}) (1-r^{2k}) &=
1 + r^{2k-1}(1-r-r^{2k}) \\ & \geq 1 + r^{2k-1}(1-r-r^2) \\
& > 1 + r^{2k-1}(1-e^{-\pi}-e^{-2\pi}) \\
& = 1 + r^{2k-1} \times 0.9549... \ > \ 1.
\end{align*}
Also, when $y>1$, both $ \frac{(4k-2)r^{2k-1}}{1+r^{2k-1}}$ and
$\frac{4k r^{2k}}{1-r^{2k}}$ are decreasing with respect to $k$. Hence
\begin{align*}
  \frac{(2k-1)(1-r^{2k-1})}{1+r^{2k-1}}
  - \frac{2k (1+r^{2k})}{1-r^{2k}}
  & = -1 - \frac{(4k-2)r^{2k-1}}{1+r^{2k-1}} - \frac{4k r^{2k}}{1-r^{2k}} \\
  & \geq  -1 - \frac{2 r}{1+r} - \frac{4 r^2}{1-r^2} \\
  & >   -1 - \frac{2 e^{-\pi}}{1+e^{-\pi}} - \frac{4 e^{-2\pi}}{1-e^{-2\pi}} \\
  &= - \Big ( \frac{1+e^{-\pi}}{1-e^{-\pi}} \Big ).
\end{align*}
One estimates
\begin{align}
  \sum_{k=1}^\infty d_{kk} & > - 16\pi^5 \Big ( \frac{1+e^{-\pi}}{1-e^{-\pi}} \Big )
  \sum_{k=1}^\infty (2k-1)^2 (2k)^2 r^{4k-1} \nonumber \\
  & = - 64 \pi^5 \Big ( \frac{1+e^{-\pi}}{1-e^{-\pi}} \Big )
  \frac{r^3 (1+31r^4+55r^8+9r^{12})}{(1-r^4)^5} \label{dkk-2}
  \end{align}
with the help of the summation formula
\begin{equation}
  \label{sum-2}
  \sum_{k=1}^\infty (2k)^2 (2k-1)^2 r^{4k-1}
  = \frac{r^2}{16} \Big ( r \Big ( r^{-1} \Big ( r \Big ( \frac{1}{1-r^4}
  \Big )_r \Big )_r \Big )_r \Big )_r
   = \frac{4 r^3 (1+31r^4+55r^8+9r^{12})}{(1-r^4)^5}.
  \end{equation}

Next consider the double sum on the right of \eqref{d-3}. Dropping
the first term in the second parenthesis pair of \eqref{T-2} one obtains
\begin{equation}
  \label{d-4}
  \sum_{k=1}^\infty \sum_{n=k+1}^\infty d_{kn} >
  \sum_{k=1}^\infty \sum_{n=k+1}^\infty
  \frac{16 \pi^5 (2k-1)^2 (2n)^2 r^{2n+2k-1}}{(1+r^{2k-1})^2 (1-r^{2n})^2}
  \Big (- \frac{2n (1+r^{2n})}{1-r^{2n}} \Big ).
  \end{equation}
For $n \geq k+1$,
\begin{align*}
  (1+r^{2k-1}) (1-r^{2n}) &= 1 + r^{2k-1} (1-r^{2n-2k+1}-r^{2n}) \\
  & > 1+ r^{2k+2} \ \geq \ 1+ r^{2n}. 
\end{align*}
Consequently
\begin{align*}
  \frac{1+r^{2n}}{(1+r^{2k-1})^2 (1-r^{2n})^3} & <
  \frac{1}{(1+r^{2k-1}) (1-r^{2n})^2} \\
  & = \frac{1}{1+r^{2k-1} \big ( (1-r^{2n})^2 +r^{2n-2k+1}(-2+r^{2n})   \big )} \\
  & \leq \frac{1}{1+r^{2k-1} \big ( (1-r^2)^2 - 2r^3   \big )} \ < \ 1.
  \end{align*}
Return to \eqref{d-4} to deduce
\begin{align}
  \sum_{k=1}^\infty \sum_{n=k+1}^\infty d_{kn} & >
  - \sum_{k=1}^\infty \sum_{n=k+1}^\infty 16 \pi^5 (2k-1)^2 (2n)^3 r^{2n+2k-1}
   \nonumber \\
 & \geq  - \sum_{k=1}^\infty \sum_{n=2}^\infty 128 \pi^5 (2k-1)^2 n^3 r^{2n+2k-1}
    \nonumber \\
   & = - \sum_{k=1}^\infty 128 \pi^5 (2k-1)^2 r^{2k}
   \Big (\frac{r^3(8-5r^2+4r^4-r^6)} {(1-r^2)^4} \Big ) \nonumber \\
   & \geq  - \sum_{k=1}^\infty 128 \pi^5 (2k-1)^2 r^{2k+3}
   \Big ( \frac{8}{(1-e^{-2\pi})^4 }   \Big ) \nonumber \\
   & = - \frac{1024 \pi^5}{(1-e^{-2\pi})^4 }
    \Big ( \frac{r^5(1+6r^2+r^4)}{(1-r^2)^3} \Big ). \label{d-6}
  \end{align}
We have used the summation formulas
\begin{align}
  \sum_{n=2}^\infty n^3 r^{2n-1}
  &= \frac{1}{8} \Big ( r \Big (  r \Big( \frac{1}{1-r^2}
      \Big )_r \Big )_r \Big )_r - r
   = \frac{r^3(8-5r^2+4r^4-r^6)} {(1-r^2)^4} \label{sum-3} \\
     \sum_{k=1}^\infty  (2k-1)^2 r^{2k+3} &=
        r^5 \Big ( r \Big ( \frac{r}{1-r^2} \Big )_r \Big )_r
   =   \frac{r^5(1+6r^2+r^4)}{(1-r^2)^3}  \label{sum-4}
\end{align}
to reach the third line and the last line respectively.

By \eqref{T-c-d}, \eqref{c-2}, \eqref{d-3} \eqref{dkk-2} and \eqref{d-6},
taking two terms from $\sum_{k=1}^\infty c_k$ and using $1<y<\sqrt{3}$, we find
\begin{align}
  T(y) & \geq \frac{4\pi^2 r}{y^2(1+r)^2}
  \Big ( \frac{\pi (1-r)}{1+r} - \frac{2}{y}  \Big )
  + \frac{36\pi^2 r^3}{y^2(1+r^3)^2}
  \Big ( \frac{3 \pi (1-r^3)}{1+r^3} - \frac{2}{y}  \Big ) \nonumber
  \\ & \qquad - 64 \pi^5 \Big ( \frac{1+e^{-\pi}}{1-e^{-\pi}} \Big )
  \frac{r^3 (1+31r^4+55r^8+9r^{12})}{(1-r^4)^5}
  - \frac{1024 \pi^5}{(1-e^{-2\pi})^4 }
  \Big ( \frac{r^5(1+6r^2+r^4)}{(1-r^2)^3} \Big ) \nonumber 
  \\ & > \frac{4\pi^2 r}{y^2(1+r)^2}
  \Big ( \frac{\pi (1-r)}{1+r} - \frac{2}{y}  \Big )
  + \frac{36\pi^2 r^3}{3(1+e^{-3\pi})^2}
  \Big ( \frac{3 \pi (1-e^{-3\pi})}{1+e^{-3\pi}} - 2  \Big ) \nonumber
  \\ & \qquad - 64 \pi^5 \Big ( \frac{1+e^{-\pi}}{1-e^{-\pi}} \Big )
  \frac{r^3 (1+31e^{-4\pi}+55e^{-8\pi} +9e^{-12\pi})}{(1-e^{-4\pi})^5}
  \nonumber \\ & \qquad - \frac{1024 \pi^5}{(1-e^{-2\pi})^4 } 
  \Big ( \frac{r^3 e^{-2\pi}(1+6e^{-2\pi}+e^{-4\pi})}{(1-e^{-2\pi})^3} \Big )
  \nonumber \\  & = \frac{4\pi^2 r}{y^2(1+r)^2}
  \Big ( \frac{\pi (1-r)}{1+r} - \frac{2}{y}  \Big ) + A r^3 \label{T-3}
\end{align}
where
\begin{align}
  A &= \frac{36\pi^2}{3(1+e^{-3\pi})^2}
  \Big ( \frac{3 \pi (1-e^{-3\pi})}{1+e^{-3\pi}} - 2  \Big ) \nonumber
  \\ & \qquad - 64 \pi^5 \Big ( \frac{1+e^{-\pi}}{1-e^{-\pi}} \Big )
  \frac{(1+31e^{-4\pi}+55e^{-8\pi} +9e^{-12\pi})}{(1-e^{-4\pi})^5} \nonumber
  \\ & \qquad - \frac{1024 \pi^5}{(1-e^{-2\pi})^4 } 
  \Big ( \frac{ e^{-2\pi}(1+6e^{-2\pi}+e^{-4\pi})}{(1-e^{-2\pi})^3} \Big )
  \nonumber
  \\ & = -21,077.61... \label{constant-A}
\end{align}
Continuing from \eqref{T-3}, one has
\begin{align}
  T(y) & > \frac{\pi^2 r^2}{(1+r)^3}
  \Big ( \Big ( \frac{4\pi}{y^2} - \frac{8}{y^3}  \Big ) r^{-1}
  - \Big (\frac{4\pi}{y^2} + \frac{8}{y^3} \Big )
  + \frac{A r (1+r)^3}{\pi^2}   \Big ).
\end{align}
Bound the last term by
\begin{equation}
  \frac{A r (1+r)^3}{\pi^2} \geq \frac{A e^{-\pi \beta}
    (1+e^{-\pi \beta})^3}{\pi^2}
  \label{boundby}
\end{equation}
and define
\begin{equation}
  \label{nu}
   \nu(y) =  \Big ( \frac{4\pi}{y^2} - \frac{8}{y^3}  \Big ) e^{\pi y}
   - \Big (\frac{4\pi}{y^2} + \frac{8}{y^3} \Big )
   + \frac{A e^{-\pi \beta} (1+e^{-\pi \beta})^3}{\pi^2}
  \end{equation}
so that
\begin{equation}
  \label{T-4}
  T(y) > \frac{\pi^2 r^2}{(1+r)^3} \nu(y).
\end{equation}
Regarding $\nu(y)$, one finds
\begin{equation}
  \label{nu'}
  \nu'(y) = e^{\pi y} y^{-3} \Big ( 4\pi^2 \Big (
  \sqrt{y} - \frac{2}{2\sqrt{y}}\Big )^2 + \frac{8}{y} \Big )
   + \frac{8\pi}{y} + \frac{24}{y^4} > 0.
  \end{equation}
Then \eqref{T-4} implies
\begin{equation}
  \label{T-5}
  T(y) >  \frac{\pi^2 r^2}{(1+r)^3} \nu(\beta)
\end{equation}
where
\begin{equation}
  \label{nu-eta}
  \nu(\beta) =  \Big ( \frac{4\pi}{\beta^2} - \frac{8}{\beta^3}  \Big ) e^{\pi \beta}
   - \Big (\frac{4\pi}{\beta^2} + \frac{8}{\beta^3} \Big )
   + \frac{A e^{-\pi \beta} (1+e^{-\pi \beta})^3}{\pi^2}.
  \end{equation}
Therefore if we can find $\beta$ so that $\nu(\beta)>0$, namely
\begin{equation}
  \label{eta-4}
     \Big ( \frac{4\pi}{\beta^2} - \frac{8}{\beta^3}  \Big ) e^{\pi \beta}
   - \Big (\frac{4\pi}{\beta^2} + \frac{8}{\beta^3} \Big )
   + \frac{A e^{-\pi \beta} (1+e^{-\pi \beta})^3}{\pi^2} >0. 
\end{equation}
then
\begin{equation}
  \label{large}
  T(y) >0, \ \ y \in [\beta, \sqrt{3}]
  \end{equation}

In summary, to invoke \eqref{small} and \eqref{large} one must choose
$\beta$ so that
\eqref{eta-1}, \eqref{eta-2}, \eqref{eta-3}, \eqref{eta-4} all hold.
Our choice is
\begin{equation}
  \beta = 1.08
\end{equation}
at last. One readily checks the four conditions.
\begin{align}
  & -\frac{1}{\beta^2} + \frac{4\pi^2 e^{-\pi \beta}}{(1+e^{-\pi \beta})^2}
    - \frac{16\pi^2 e^{-2 \pi \beta}}{(1-e^{-2\pi \beta})^2}  
    \ \Big | _{\beta =1.08} = 0.2058... >0; \\
    & -1 + 48 e^{-\pi \beta} - 312  e^{-2\pi \beta} \ \Big |_{\beta=1.08}
    =  0.2608... >0;  \\
   &       -\frac{6}{4\pi^4 \beta^4} +
    \frac{1}{(1-e^{-2\pi})^4} \Big ( e^{-\pi \beta}- 24 e^{-2\pi \beta} +
    104 e^{-3\pi \beta} \Big )\ \Big |_{\beta=1.08} = -0.0007930... <0; \\
    &      \Big ( \frac{4\pi}{\beta^2} - \frac{8}{\beta^3}  \Big ) e^{\pi \beta}
   - \Big (\frac{4\pi}{\beta^2} + \frac{8}{\beta^3} \Big )
   + \frac{A e^{-\pi \beta} (1+e^{-\pi \beta})^3}{\pi^2} \ \Big |_{\beta =1.08}
  = 35.20...>0. 
\end{align}
The proof of \eqref{quo-mono} is complete.
\end{proof}

\

\noindent {\bf Acknowledgment}. S. Luo and J. Wei are partially
supported by NSERC-RGPIN-2018-03773;
X. Ren is partially supported by NSF grant DMS-1714371.

\bibliography{citation}
\bibliographystyle{plain}

\end{document}